\documentclass[reqno]{amsart}
\usepackage{url}
\usepackage{amssymb,amsmath}
\usepackage{graphicx}

\usepackage[a4paper, margin=2.5cm]{geometry}
\input xypic
\xyoption{all}

\usepackage{times}
\usepackage{amsmath, amsthm, amsfonts}

\usepackage[colorlinks,linkcolor=red,citecolor=green,anchorcolor=blue]{hyperref}
\usepackage{color}
\usepackage{tikz}
\usetikzlibrary{arrows} \usetikzlibrary{patterns}
\usepackage[all,cmtip]{xy}

\input diagxy

\newtheorem{theorem}{Theorem}[section]
\newtheorem{lem}[theorem]{Lemma}
\newtheorem{defn}[theorem]{Definition}
\newtheorem{cor}[theorem]{Corollary}
\newtheorem{prop}[theorem]{Proposition}
\newtheorem{construction}[theorem]{Construction}

\theoremstyle{definition}
\newtheorem{exa}[theorem]{Example}

\newtheorem{rem}[theorem]{Remark}

\numberwithin{equation}{section}

\def \msf{\mathsf}

\def \mc{\mathcal}

\def \inv{^{-1}}
\def \0{\infty}

\def \qq{\quad}
\def \rto{\rightarrow}
\def \Rto{\Rightarrow}
\def \hrto{\hookrightarrow}
\def \rrto{\rightrightarrows}
\def \Hom{\mathrm{Hom}}
\def \ctimes{\stackrel{\circ}{\times}}
\def \tcirc{\tilde{\circ}}

\def \rhpu{\rightharpoonup}

\def \A{{\sf A}}
\def \B{{\sf B}}

\def \G{{\sf G}}
\def \H{{\sf H}}
\def \I{{\sf I}}
\def \J{{\sf J}}
\def \K{{\sf K}}
\def \L{{\sf L}}
\def \M{{\sf M}}
\def \N{{\sf N}}

\def \Q{{\sf Q}}

\def \U{{\sf U}}

\def \W{{\sf W}}

\def \f{{\sf f}}
\def \g{{\sf g}}

\def \u{{\sf u}}
\def \v{{\sf v}}
\def \w{{\sf w}}
\def \i{{\sf i}}

\def \q{{\sf q}}

\allowdisplaybreaks

\begin{document}



\title{The groupoid structure of groupoid morphisms}


\author{Bohui Chen}
\address{Department of Mathematics and Yangtze Center of Mathematics, Sichuan University, 610065, Chengdu, China}
\email{bohui@cs.wisc.edu}
\author{Cheng-Yong Du}
\address{Department of Mathematics, Sichuan Normal University, 610068, Chengdu, China}
\email{cyd9966@hotmail.com}
\author{Rui Wang}
\address{Department of Mathematics, University of California, Irvine, CA 92697-3875, USA}
\address{Department of Mathematics, University of California, Berkeley, CA, 94720--3840, USA}
\email{ruiwang@berkeley.edu}


\begin{abstract}
In this paper we construct two groupoids from
morphisms of groupoids, with one from a categorical viewpoint
and the other from a geometric viewpoint.
We show that for each pair of groupoids, the two kinds of
groupoids of morphisms are equivalent.
Then we study the automorphism groupoid of a groupoid.
\end{abstract}

\maketitle


\section{Introduction}

In this paper we study morphisms and automorphisms of groupoids. Our motivation comes from the study of maps between orbifolds and group actions on orbifolds. It is well-known that an orbifold can be considered as a Morita equivalence class of proper \'etale Lie groupoids (cf. \cite{Adem-Leida-Ruan2007}). Hence, among various definitions for morphisms between orbifold groupoids, it turns out that the orbifold homomorphism (cf. \cite{Adem-Leida-Ruan2007}) is most adapted to the purpose of understanding maps between orbifolds. The orbifold homomorphism is the motivation of the morphism we consider in this paper (cf. Definition \ref{D morphism}). Further more, motivated by the study of the space of the orbifold homomorphisms in orbifold Gromov--Witten theory \cite{Chen-Ruan2002,Chen-Ruan2004}, it is important to build up a groupoid structure for morphisms of groupoids. This is the main issue we deal with in this paper.

We now outline our approach.

\subsection{Morphism groupoids}
Given a pair of groupoids $\G$ and $\H$, by a morphism (cf. Definition \ref{D morphism}) we mean a pair of strict morphisms
\begin{align}\label{E generalized-morphisms-introduction}
\xymatrix{
\sf G  & \sf K \ar[l]_-\psi \ar[r]^-{\u} & \H}
\end{align}
with $\psi$ being an equivalence of groupoids. We denote the morphism by $\sf (\psi,K,u):G\rhpu H$.

An arrow between two such morphisms is captured by the following diagram, with $\alpha$ being a natural transformation for the diagram on the right:
\begin{align}\label{E arrow-introduction}
\begin{split}
\begin{tikzpicture}
\def \x{3}
\def \y{1}
\node (A00) at (0,0)       {$\sf G$};
\node (A10) at (\x,0)      {$\K_1\times_\G\K_2$};
\node (A11) at (\x,\y)     {$\K_1$};
\node (A1-1) at (\x,-1*\y) {$\K_2$};
\node (A30) at (3*\x,0)    {$\sf H$.};
\node at (1.8*\x,0) {$\Downarrow \alpha$};
\path (A00) edge [<-] node [auto] {$\scriptstyle{\psi_1}$} (A11);
\path (A00) edge [<-] node [auto,swap] {$\scriptstyle{\psi_2}$} (A1-1);
\path (A11) edge [->] node [auto] {$\scriptstyle{\u_1}$} (A30);
\path (A1-1) edge [->] node [auto,swap] {$\scriptstyle{\u_2}$} (A30);
\path (A10) edge [->] node [auto] {$\scriptstyle{\pi_1}$} (A11);
\path (A10) edge [->] node [auto,swap] {$\scriptstyle{\pi_2}$} (A1-1);
\end{tikzpicture}
\end{split}
\end{align}
We denote the arrow by $(\psi_1,\K_1,\u_1)\xrightarrow{\alpha}(\psi_1,\K_2,\u_2)$.

We define the (vertical) composition of arrows in \S\ref{Subs MgVfp} (cf. Construction \ref{C vert-compose}). The main ingredient in the construction is the fiber product of groupoids. Then we get a groupoid of morphisms (cf. Theorem \ref{T Mor-groupoid})
\[
{\sf Mor(G,H)}=(\mathrm{Mor}^1(\G,\H)\rrto \mathrm{Mor}^0(\G,\H)).
\]
In \S \ref{Subs MgVsfp}, we replace the fiber product by strict fiber product (cf. Definition \ref{D Geo-fib}) to simplify the constructions. For this purpose, we focus on full-morphisms (a morphism $(\psi,\sf K, u)$ is called a full-morphism if $\psi^0$ is surjective). With the same procedure as in \S\ref{Subs MgVfp}, we get a groupoid of full-morphisms (cf. Theorem \ref{T FMor-groupoid})
\[
{\sf FMor(G,H)}=(\mathrm{FMor}^1(\G,\H)\rrto \mathrm{FMor}^0(\sf G,H)).
\]
We show in Theorem \ref{T equi-i} that there is a natural equivalence of groupoids $\sf i:FMor(G,H)\rto Mor(G,H)$.

In \S \ref{S Composition-Mor-FMor}, we consider the composition functors:
$$
\circ: \sf{Mor}(G,H)\times \sf{Mor}(H,N)\rto \sf{Mor}(G,N),
\qq
\text{and}
\qq \tilde\circ: \sf{FMor}(G,H)\times \sf{FMor}(H,N)\rto \sf{FMor}(G,N).
$$

For instance, the composition of $\sf(\psi,K,u):\sf G\rhpu H$ and $\sf(\phi,L,v): H\rhpu N$, denoted by $\sf (\phi, L,v)\circ(\psi,K,u)$, is given by
\begin{align*}
\xymatrix{
\G &\K\times_\H\L \ar[l]_-{\psi\circ\pi_1} \ar[r]^-{\v\circ\pi_1} &\N.}
\end{align*}
Furthermore, we study the horizontal composition of arrows in \eqref{E arrow-introduction} to get the functor $\circ$. We show that these compositions are associative (modulo certain canonical isomorphisms).

As an application, we study the automorphisms of $\sf G$ in \S\ref{S gpaction}. An automorphism of $\sf G$ is defined as a morphism $\sf (\psi,K,u):G\rhpu G$, such that after composing with another morphism $\sf (\phi,L,v):G\rhpu G$ there are arrows
\[
\sf(\psi,K,u)\circ(\phi,L,v)\xrightarrow{\alpha} 1_G,\qq\mbox{and}\qq
\sf(\phi,L,v)\circ (\psi,K,u)\xrightarrow{\beta} 1_G,
\]
where
\[
\sf 1_G=(id_G,G,id_G):
\xymatrix{\G &\G\ar[l]_{\sf id_G} \ar[r]^{\sf id_G}& \G}.
\]
Denote by $\mathrm{Aut}^0(\sf G)$ the set of all automorphisms of $\sf G$. We restrict $\sf Mor(G,G)$ to $\mathrm{Aut}^0(\sf G)$ to get a groupoid of automorphisms $\sf Aut(G)$ of $\G$. We show that the coarse space $|\sf Aut(G)|$ is a group, and the automorphism groupoid $\sf Aut(G)$ is a $\mc K(\sf G)$-gerbe over $|\sf Aut(G)|$ (cf. Theorem \ref{T |aut|-is-group}). As an application we introduce a definition of group action on a groupoid $\sf G$ with trivial $\mc K(\G)$ (see Definition \ref{D group-action}).

\subsection{Relation with other literatures}
Moderijk--Pronk (\cite{Moerdijk-Pronk1997}) found the relation between effective (or reduced) orbifolds and effective orbifold groupoids, that is an orbifold corresponds to a Morita equivalence class of effective orbifold groupoids. So there are two way to study effective orbifolds, one by effective orbifold groupoids, the other one by orbifold charts/altases. See also for example \cite{Adem-Leida-Ruan2007,Lupercio-Uribe2004} the relation between orbifold atlases and orbifold groupoids. As the category of manifolds people also want to get a category of orbifolds. However, it turns out that 2-category and bicategory appear naturally.

On the groupoid side, there is a 2-category {\bf 2Gpd} of groupoids (or topological groupoids, or Lie groupoids or etc.), with morphism categories consisting of strict morphisms and natural transformations. Pronk (\cite{Pronk1996}) proposed a method, called right bicalculus of fractions, to construct a new bicategory $\mc A[W\inv]$ out of a bicategory $\mc A$ from a collection $W$ of morphisms in $\mc A$ that satisfies various axioms. By applying right bicalculus of fractions to {\bf 2Gpd} with $W$ being the collection of all equivalences of groupoids, one could get a bicategory ${\bf 2Gpd}[W\inv]$ of groupoids. Tommasini revisited this construction in \cite{Tommasini2016a} and constructed a bicategory of effective orbifold groupoids with $W$ consisting of equivalence of effective orbifold groupoids in \cite{Tommasini2016b}. The 1-morphisms in the resulting bicategory corresponds to our morphisms \eqref{E generalized-morphisms-introduction}. In a bicategory obtained via right bicalculus of fractions, the compositions of 1-morphisms and 2-morphisms depend on certain choices. In particular, although different choices for the compositions of 2-morphisms will lead to the same bicategory when the choices for the compositions of 1-morphism are fixed, different choices for the compositions of 1-morphisms would lead to different resulting bicategories, even if they are equivalent. In this paper, we construct the morphism groupoids and write down the composition functor explicitly. Hence we have explicit formulae for compositions of 1-morphisms and 2-morphisms. On the other hand, the strict fiber product gives a more geometric way to compose full-morphisms. This would be useful for us to assign smooth structure over the morphism groupoids when we deal with Lie groupoids. Moreover, we also show that the (horizontal) composition functors are associative under canonical isomorphisms between fiber products. This implies that what we get are two bicategories with morphism groupoids being $\sf Mor(G,H)$ and $\sf FMor(G,H)$ respectively.

On the orbifold charts/atlases side, Pohl (\cite{Pohl2017}) studied the category of effective (or reduced) orbifolds by using local charts/orbifold atlas, Borzellino--Brunsden (\cite{Borzellino-Brunsden2008}) and Schmeding (\cite{Schmeding2015}) studied the group of orbifold diffeomorphisms of an orbifold, where they viewed orbifold morphisms as equivalence classes of collections of local liftings that satisfy some certain compatible conditions, which in fact corresponds to the coarse space $|\sf Aut(G)|$ of our construction. There is also a 2-category of effective orbifold atlases (cf. \cite{Tommasini2012,Du-Shen-Zhao2018}), which in fact corresponds to the 2-category of effective orbifold groupoids, a sub-category of {\bf 2Gpd}. Tommasini (\cite{Tommasini2016b}) also constructed a bicategory of effective orbifold atlases by modifying the construction of the 2-category of effective orbifold atlases, followed his construction of bicategory of effective orbifold groupoids.

The groupoid structure over $|\sf Mor(G,H)|$ constructed here gives a more concrete description of orbifold morphisms and would be useful in the study of moduli spaces of pseudo-holomorphic curves in orbifold  Gromov--Witten theory. On the other hand, non-effective orbifolds appear naturally in orbifold Gromov--Witten theory. However, there is no easy way to describe non-effective orbifolds. It is better to study non-effective orbifolds via orbifold groupoids. Therefore we study the morphisms between orbifolds via groupoids in this paper.

\subsection{Topological and Lie groupoids}\label{Subs discussion-on-top-lie-groupoids-1.3}
All discussions can be easily generalized to the cases of topological groupoids and Lie groupoids. For topological groupoids, we only need to add the continuous conditions on various maps involved. For Lie groupoids we need first add smooth conditions on various maps involved, and then replace surjective maps by surjective submersions.

When we construct $\sf FMor(G,H)$ we could also consider a more restrictive full-morphisms. For example when we consider topological/Lie groupoids, we could require that in every full-morphism
$\sf (\psi,K,u):G\rhpu H$
the $\psi^0: K^0\rto G^0$ is an open covering of $ G^0$, and $\sf K$ is the pull-back groupoid over $ K^0$ via $\psi^0$. Under this constraint we could make $\sf FMor(G,H)$ into a topological groupoid. We discuss this issue in the appendix.

\subsection*{Acknowledgements}
The authors thank Lili Shen and Xiang Tang for helpful discussions. This work was supported by the National Natural Science Foundation of China (No. 11431001 $\&$ No. 11726607 $\&$  No. 11890663 $\&$ No. 11501393), by the IBS project ($\#$IBS-R003-D1), by Sichuan Science and Technology Program (No. 2019YJ0509) and by Sichuan University (No. 1082204112190).

\section{Basic concepts of groupoids}

For basic concepts about groupoids we refer readers to
\cite{Adem-Leida-Ruan2007,Mackenzie1987,Moerdijk-Mrcun2003}.

\subsection{Groupoids}
\label{Subs Top-Groupoid}

Let $\sf G$ be a small category with the set of objects denoted by $G^0$ and the set of morphisms denoted by $G^1$. Here $G^1$ is the collection of morphisms
\[
G^1=\bigsqcup_{(a,b)\in G^0\times G^0} G^1(a,b),
\]
where $G^1(a,b)$\footnote{In literatures on category, this is denoted by $\Hom(a,b)$. In this paper, we use this notation to emphasis the groupoid structure.} is the set of morphisms from $a$ to $b$. We call a morphism $x\in G^1(a,b)$ an {\em arrow} from $a$ to $b$, and call $a$ and $b$ to be the {\em source} and the {\em target} of $\alpha$ respectively. We write
\[
a=s(x),\qq b=t(x),
\qq
\mbox{and}
\qq
x:a\rto b\qq (\mbox{or}\;\; a \xrightarrow{x} b),
\]
where $s,t: G^1\rto G^0$ are called the {\em source map} and the {\em target map} of the category $\G$ respectively. Denote the composition of arrows\footnote{In this paper we use the convention that the composition of arrows of a groupoid is going from left to right, not the usual notation of composition of maps.} by
\[
\cdot:  G^1(a,b)\times  G^1(b,c)\rto  G^1(a,c), \qq
(x,y)\mapsto x\cdot y.
\]

\begin{defn}\label{D Groupoid}
We say that $\sf G$ is a {\em groupoid} if
\begin{enumerate}
\item[(1)] for any $a\in  G^0$, there exists a {\em unit} $1_a\in G^1(a,a)$ with respect to the composition ``$\cdot$'', i.e, $1_a \cdot x=x$ and $x\cdot 1_a=x$;

\item[(2)] for any $x\in  G^1(a,b)$, there exists a unique {\em inverse} $y\in  G^1(b,a)$ such that $x\cdot y=1_a$ and $y\cdot x=1_b$. We denote $y$ by $x^{-1}$.
\end{enumerate}
Define two  maps:
\begin{itemize}
\item the unit map $u:G^0\rto  G^1,\; a\mapsto 1_a$;

\item the inverse map $i:  G^1\rto  G^1,\; x\mapsto  x^{-1}$.
\end{itemize}
Therefore a groupoid $\sf G$ is a pair of sets $( G^0,G^1)$ with structure maps $(\cdot, s,t, u,i)$. We may denote $\sf G$ by $( G^1\rrto G^0)$, where the double arrows denote the source and target maps $s$ and $t$.

If we assume that $G^0$ and $G^1$ are topological space and the structure maps are continuous, we call $\sf G$ a {\em topological groupoid}.
\end{defn}

$ G^1$ defines an equivalence relation on $ G^0$: we say that
\[
a\sim b \iff G^1(a,b)\not=\varnothing.
\]
We call the quotient space $ G^0/G^1$ to be the {\em coarse space} of $\sf G$ and  denote it by $|\sf G|$. The projection map from $ G^0$ to $|\sf G|$ is denoted by $|\cdot|: G^0\rto \sf|G|$. When $\sf G$ is a topological groupoid, $|\sf G|$ is equipped with the quotient topology.

\begin{defn}\label{D strict-morphism}
By a {\em strict morphism} from a groupoid $\G=(G^1\rrto G^0)$
to a groupoid $\H=(H^1\rrto H^0)$, we mean a functor from the category $\sf G$ to $\sf H$. We denote a strict morphism by $\f=(f^0,f^1)$ with
\[
f^0: G^0\rto H^0,\qq f^1: G^1\rto H^1.
\]
A strict morphism $\sf f: G\rto H$ is an {\em isomorphism} if it has an inverse strict morphism.
\end{defn}

For a groupoid $\G$ we denote by ${\sf id_G}=(id_{G^0},id_{G^1}):\G\rto\G$ the identity strict morphism.

\begin{defn}\label{D natural-trans}
Let $\sf f,g:G\rto H$ be two strict morphisms. A {\em natural transformation} from $\sf f$ to $\sf g$, denoted by $\sf f \stackrel{\alpha}{\Rto}g:G\rto H$ or simply by $\sf f \stackrel{\alpha}{\Rto}g$, is a natural transformation between the two functors.
\end{defn}

A strict morphism from $\sf f: G\rto  H$ induces a map  $|\sf f|: |G|\rto | H|$ on coarse spaces. If there is an $\sf f \stackrel{\alpha}{\Rto}g:G\rto H$, then $\sf |f|=|g|$.

\subsection{Equivalence of groupoids}

\begin{defn} \label{D equi&Moritaequi}
Let $\sf G$ and $\sf H$ be two groupoids. A strict morphism $\sf f: G\rto H$ is called an {\em equivalence} if
\begin{enumerate}
\item[(1)]  the map $t\circ \mathrm{proj}_2: G^0 \times_{f^0,H^0,s} H^1 \xrightarrow{\mathrm{proj}_2} H^1\xrightarrow{t} H^0$ is surjective;

\item[(2)]  the square
\[
\xymatrix{
 G^1\ar[rr]^{ f^1} \ar[d]_{s\times t} &&
 H^1\ar[d]^{s\times t}\\
 G^0\times G^0\ar[rr]^{ f^0\times f^0} &&
 H^0\times H^0.
}\]
is a fiber product.
\end{enumerate}

Two groupoids $\sf G$ and $\sf H$ are {\em Morita equivalent} if there is a third groupoid $\sf K$ and two equivalences
\[
\xymatrix{\sf G &\sf K\ar[r]^-\psi\ar[l]_-{\phi} &\H.}
\]
\end{defn}

\begin{rem}\label{R equiv-bij-arrows}
We have the following two simple facts about the definition of equivalence.

(1) The first condition means that $\sf f$ is {\em essentially surjective}.

(2) The second condition means that $\sf f$ is {\em full and faithful}, that is for any $a,b\in G^0$, $ f^1$ induces a bijection
\begin{align}\label{E equiv-bij-arrows}
f^1: G^1(a,b)\rto H^1( f^0(a), f^0(b)).
\end{align}
Consequently, consider three arrows $x,y,z\in  H^1$ with $z=x\cdot y$, i.e. they fit into a commutative diagram
\[
\xymatrix{a\ar[r]^x \ar[dr]_z & b\ar[d]^y \\
&c,}
\]
and $a',b',c'\in  G^0$ such that $f^0(a')=a$, $f^0(b')=b$, $f^0(c')=c$. Then from \eqref{E equiv-bij-arrows} we get a commutative diagram in $G^1$
\begin{align}\label{E f-inv-is-homo}
\xymatrix{a'\ar[rr]^{( f^1)\inv(x)} \ar[drr]_{( f^1)\inv(z)} && b' \ar[d]^{( f^1)\inv(y)} \\ &&c'. }
\end{align}
\end{rem}

We state two simple facts without proofs.

\begin{lem}\label{L f-equiv=g-equiv}
Given a natural transformation $\sf f\stackrel{\alpha}{\Rto} g:G\rto H$, if one of $\sf f$ and $\sf g$ is an equivalence then so is the other one.
\end{lem}

\begin{lem}\label{L eqs-compose-eq}
Given a pair of equivalences $\sf G\xrightarrow{\psi} H\xrightarrow{\phi} N$, the composition $\phi\circ\psi:\sf G\rto N$ is also an equivalence.
\end{lem}

In the following we use $\psi,\phi,\varphi,\ldots$ to denote equivalences, and use $\sf f,g,h,u,v,w$, $\ldots$ to denote general strict morphisms.

On the other hand we have the following obvious criterion for the equivalence between a special kind of subgroupoid\footnote{Here we do not need the precise definition of sub-groupoid. One can think it as a subcategory. For explicit definition of subgroupoid see \cite[Definition 2.4 in Chapter 1]{Mackenzie1987}.} and the ambient one.

\begin{lem}\label{L inclusion-equiv}
Suppose $ G^0\subseteq H^0$ and $\G:=\H|_{G^0}$ is the restriction\footnote{This restriction groupoid $\G:=\H|_{G^0}$ has object space $G^0$ and morphism space $\bigsqcup_{a,b\in G^0}H^1(a,b)\subseteq H^1$. The structure maps are inherited from $\sf H$ naturally.} of $\sf H$ on $G^0$. If under the map $|\cdot|: H^0\rto|\H|$ we have $|G^0|=|H^0|=|\sf H|$, then the natural inclusion $\sf i: G\hrto H$ is an equivalence.
\end{lem}

\subsection{Fiber product}\label{Subs fib-prdct}

Let $\sf f: F\rto H$ and $\sf g: G\rto H$ be two strict morphisms. The {\em fiber product} $\sf F\times_{f,H,g} G$ (or simply $\sf F\times_H G$) is defined to be a groupoid as the following (cf. \cite{Adem-Leida-Ruan2007}):
\begin{enumerate}
\item The object space is
\[
({\sf F\times_H G})^0=F^0\times_{ f^0, H^0,s}  H^1
\times_{t, H^0, g^0}  G^0=
 F^0\times_{H^0} H^1   \times_{H^0} G^0.
\]
An object is a triple $(a,x, b)$, with $a\in  F^0, b\in  G^0$ and $x\in{ H^1(f}^0(a), g^0(b))$. We draw it as
\begin{align}\label{E obj-fib-prdct-1}
\xymatrix{a \ar@{.>}[r]^x & b}
\qq \mbox{or}\qq
\xymatrix{a \ar@{.>}[r]^x_{\sf H} & b.}
\end{align}

\item Given two objects $(a,x,b)$ and $(a',x',b')$, an arrow from $(a,x,b)$ to $(a',x',b')$ consists of a pair of arrows $(y,z)$ with $y\in F^1(a,a'), z\in G^1(b,b')$, such that $ x\cdot g^1(z)= f^1(y)\cdot x'$, i.e. we have the following commutative diagrams
\begin{align}\label{E arrow-in-F*/H*G}
\begin{split}
\xymatrix{
a\ar@{.>}[r]^x\ar[d]_y  &
b\ar[d]^z\\
a'\ar@{.>}[r]^{x'} & b'
}
\end{split}\end{align}
after all arrows are transferred into $H^1$. Hence the arrow space is
\[
({\sf F\times_H G})^1= F^1\times_{s,F^0,\text{proj}_1}
\left(F^0\times_{H^0}\times H^1\times_{H^0} G^0\right)
\times_{\text{proj}_3,G^0,s}  G^1.
\]
Denote an arrow by $(y,(a,x,b),z)$. The source and target maps are obvious from the diagram \eqref{E arrow-in-F*/H*G}.

\item All other structure maps are obvious.
\end{enumerate}
There are two natural strict morphisms, called {\em projections}:
\begin{align*}
&\pi_1: {\sf F\times_H G\rto F},\qq (a,x,b)\mapsto a,\qq (y,(a,x,b),z)\mapsto y,
\\
&\pi_2: {\sf F\times_H G\rto G},\qq (a,x,b)\mapsto b,\qq (y,(a,x,b),z)\mapsto z.
\end{align*}
It is known that $\f\circ \pi_1$ and $\g\circ\pi_2$ are different up to a natural transformation.

We have the following useful result, which can be verified straightforwardly.

\begin{lem} \label{L G*/HK-equiv G}
When $\sf g:G\rto H$ is an equivalence, $\pi_1: \sf F\times_H G\rto F$ is an equivalence. Similarly, when $\sf f:F\rto H$ is an equivalence, $\pi_2: \sf F\times_H G\rto G$ is an equivalence.
\end{lem}

\subsection{Strict fiber product}\label{Subs geo-fib-prdct}

In this paper we will also consider a simpler version of fiber product which we call it strict fiber product. Lemma \ref{L Go*/HK-eq-G*/HK} explains that under certain conditions, we may replace fiber products by strict fiber products.

\begin{defn}\label{D Geo-fib}
Given two strict morphisms $\f_i:\G_i=( G_i^1\rrto G_i^0)\rto\H=(H^1\rrto H^0)$ for $i=1,2$, we define the {\em strict fiber product} ${\G_1 \ctimes_{\f_1,\H,\f_2} \G_2}$ (or simply ${\G_1 \ctimes_\H \G_2}$) as
\[
{\G_1 \ctimes_\H \G_2}:= (G^1_1\times_{H^1} G_2^1\rrto G^0_1\times_{H^0} G^0_2),
\]
where
\begin{align*}
G^0_1\times_{H^0} G^0_2=\{(a,b)\mid f_1^0(a)= f_2^0(b)\in  H^0\}, \qq \text{and} \qq G^1_1\times_{H^1} G_2^1&=\{(x,y)\mid f_1^1(x)= f^1_2(y)\in  H^1\}
\end{align*}
are fiber products of sets.
\end{defn}

In the following, we will some times also write an object in ${ G^0_1\times_{H^0} G^0_2}$ in the way as \eqref{E obj-fib-prdct-1}. For example
\[
\xymatrix{a\ar@{.>}[r]^{1_{f^0_1(a)}}_{\sf H} & a'}
\in  G^0_1\times_{H^0} G^0_2.
\]
There are natural strict morphisms, also called {\em projections}:
\begin{align*}
\tilde \pi_1:{\G_1 \ctimes_\H \G_2\rto \G_1},&\qq
(a,b)\mapsto a,\qq (x,y)\mapsto x
\\
\tilde \pi_2:{\G_1 \ctimes_\H \G_2\rto \G_2},&\qq
(a,b)\mapsto b,\qq (x,y)\mapsto y.
\end{align*}
It is clear that $\f_1\circ\tilde\pi_1=\f_2\circ \tilde \pi_2$.

There is an injective strict morphism connecting these two kinds of fiber products
\begin{eqnarray}\label{E q-geo-fbi-to-fib}
\q=(q^0,q^1):\G_1 \ctimes_\H \G_2&\rto& \G_1\times_\H \G_2,\nonumber\\
(a,b)&\mapsto&(a,1_{ f_1^0(a)},b),\\
(x,y)&\mapsto&(x,(s(x),1_{ f^0_1(s(x))},s(y)),y).\nonumber
\end{eqnarray}
Set
\[
U^0:=\mathrm{Im}\, q^0\subseteq (\G_1\times_\H \G_2)^0.
\]
Via $\sf q$ we could view $\G_1\ctimes_\H \G_2$ as a sub-groupoid of $\G_1\times_\H \G_2$. In fact we have

\begin{lem}\label{L Geofib-as-subgrpd}
${\q: \G_1\ctimes_\H\G_2\rto (\G_1\times_\H\G_2)}|_{U^0}$ is an isomorphism.
\end{lem}

\begin{proof}
Since both $q^0$ and $q^1$ are injective, we only need to show that
\[
 q^1:({\G_1\ctimes_\H\G_2})^1((a,b),(a',b'))\rto
({\G_1\times_\H\G_2})^1((a,1_{f^0_1(a)},b),(a',1_{f^0_1(a')},b'))
\]
is surjective. Suppose we have an arrow
\begin{align}\label{E arrow-in-lemma-2.11}
(x,(a,1_{f^0_1(a)},b),y):
(a,1_{f^0_1(a)},b)\rto (a',1_{f^0_1(a')},b')
\end{align}
in $(\G_1\times_\H\G_2)^1$. Then $f^1_1(x)\cdot 1_{ f^0_1(a')}=1_{ f^0_1(a)}\cdot  f_2^1(y)$, i.e. $f^1_1(x)=f_2^1(y)$. Hence we get an arrow $(x,y)$ in $\G_1\ctimes_\H \G_2$ which is a preimage of the arrow \eqref{E arrow-in-lemma-2.11}.
\end{proof}

\begin{defn}\label{D full-equiv}
An equivalence $\psi:\sf G\rto H$ is called a {\em full-equivalence} if $\psi^0$ is surjective.
\end{defn}

\begin{lem}\label{L Go*/HK-eq-G*/HK}
When one of $\msf f_1$ and $\msf f_2$ is a full-equivalence, $\q: \G_1 \ctimes_\H \G_2\rto \G_1\times_\H \G_2$ is an equivalence. In that case, $({\G_1 \times_\H \G_2})|_{U^0}$ is equivalent to $\G_1 \times_\H \G_2$.
\end{lem}

\begin{proof}
By Lemma \ref{L inclusion-equiv} we only need to show that every object $(a,x,b)\in(\G_1\times_\H \G_2)^0$ is connected to an object in $U^0$. Without loss of generality, we assume that $\f_1$ is a full-equivalence. Hence $f^0_1$ is surjective.

Take a pre-image $a'$ of $f_2^0(b)$ under $f_1^0$, i.e. $f_1^0(a')=f_2^0(b)$. Then $(a',1_{f_1^0(a')},b)\in U^0$. Since $x:f_1^0(a)\rto f_2^0(b)=f_1^0(a')$ and $\f_1$ is an equivalence, we get a unique arrow $(f_1^1)\inv(x):a\rto a'$. Then one can see that
\[
((f_1^1)\inv(x),(a,x,b),1_b):(a,x,b)\rto (a',1_{f_1^0(a')},b)
\]
is an arrow in $(\G_1\times_\H \G_2)^1$ that connects $(a,x,b)$ with $(a',1_{ f^0_1(a')},b)\in U^0$.
\end{proof}

We also have analogues of Lemma \ref{L eqs-compose-eq} and
Lemma \ref{L G*/HK-equiv G}.

\begin{lem}\label{L feqs-compose-feq}
Let $\psi:\sf G\rto H$ and $\phi:\sf H\rto N$ both be full-equivalences. Then the composition $\phi\circ\psi:\sf G\rto N$ is also a full-equivalence.
\end{lem}

\begin{lem}\label{L Go*/HK-equiv-G}
When $\msf f_2$ is a full-equivalence, $\tilde\pi_1:{\G_1 \ctimes_\H \G_2\rto \G_1}$ is a full-equivalence. When $\msf f_1$ is a full-equivalence, $\tilde\pi_2:{\G_1 \ctimes_\H \G_2\rto \G_2}$ is a full-equivalence.
\end{lem}

\subsection{Canonical isomorphisms for (strict) fiber products}

\begin{lem}\label{L fib-prdct-asso}
Given four strict morphisms $\sf f:H\rto G, g:K\rto G$, $\sf u:K\rto L$, and $\sf v:M\rto L$, we have two canonical isomorphisms and the following commutative diagram
\[\xymatrix{
(\H\ctimes_{\sf f,G,g}\K)\ctimes_{\u\circ\pi_2,\L,\v}\M \ar[r]^{\cong} \ar[d]_{\q} & \H\ctimes_{\f,\G,\g\circ \pi_1}(\K\ctimes_{\u,\L,\v}\M) \ar[d]^\q
\\
{\sf (H\times_{f,G,g}K)}\times_{\u\circ\pi_2,\L,\v}\M \ar[r]^{\cong} & \H\times_{\f,\G,\g\circ \pi_1}(\sf K\times_{u,L,v}M).}
\]
\end{lem}

\begin{lem}\label{L Gc*/GH-cong-H}
For a strict morphism $\sf f:G\rto H$, there are canonical isomorphisms $\sf H\ctimes_{id_H,H,f}G\cong G\cong  G\ctimes_{f,H,id_H}H$ given by projections.
\end{lem}

\section{Morphism groupoids}\label{S Mor}

In this section for each pair $(\sf G,H)$ of groupoids we construct two groupoids of morphisms, $\sf Mor(G,H)$ and $\sf FMor(G,H)$. Then we will show that these two groupoids are equivalent to each other.

\subsection{Morphism groupoids via fiber products}
\label{Subs MgVfp}

\begin{defn}\label{D morphism}
By a {\em morphism}\footnote{Such a morphism between orbifold groupoids is called an orbifold homomorphism (cf. \cite{Adem-Leida-Ruan2007}).} from $\sf G$ to $\sf H$, we mean two strict morphisms in the diagram
\[
\xymatrix{ \G & \K \ar[l]_-{\psi} \ar[r]^-{\u}  &\H}
\]
with $\psi$ being an equivalence. We denote such a morphism by $(\psi, \sf K,u): G\rhpu H$, and the set of morphisms from $\sf G$ to $\sf H$ by $\mathrm{Mor}^0(\sf G,H)$.
\end{defn}

\begin{defn}\label{D arrow}
Given two morphisms $(\psi_1, \K_1,\u_1):\sf G\rhpu H$ and $(\psi_2, \K_2,\u_2):\sf G\rhpu H$, an {\em arrow} $(\psi_1, \K_1,\u_1)\stackrel{\alpha}{\rto} (\psi_2, \K_2,\u_2)$ is a natural transformation $\u_1\circ\pi_1\stackrel{\alpha}{\Rto} \u_2\circ\pi_2$, i.e.
\[
\begin{tikzpicture}
\def \x{3}
\def \y{1}
\node (A00) at (0,0)       {$\sf G$};
\node (A10) at (\x,0)      {$\K_1\times_\G\K_2$};
\node (A11) at (\x,\y)     {$\K_1$};
\node (A1-1) at (\x,-1*\y) {$\K_2$};
\node (A30) at (3*\x,0)    {$\sf H$.};
\node at (1.8*\x,0) {$\Downarrow \alpha$};
\path (A00) edge [<-] node [auto] {$\scriptstyle{\psi_1}$} (A11);
\path (A00) edge [<-] node [auto,swap] {$\scriptstyle{\psi_2}$} (A1-1);
\path (A11) edge [->] node [auto] {$\scriptstyle{\u_1}$} (A30);
\path (A1-1) edge [->] node [auto,swap] {$\scriptstyle{\u_2}$} (A30);
\path (A10) edge [->] node [auto] {$\scriptstyle{\pi_1}$} (A11);
\path (A10) edge [->] node [auto,swap] {$\scriptstyle{\pi_2}$} (A1-1);
\end{tikzpicture}
\]
Denote by $\mathrm{Mor}^1(\G,\H)((\psi_1, \K_1,\u_1),(\psi_2, \K_2,\u_2))$, the set of all arrows from $(\psi_1, \K_1,\u_1)$ to $(\psi_2, \K_2,\u_2)$, and by
\[
\mathrm{Mor}^1(\G,\H):= \bigsqcup_{(\psi_i, \K_i,\u_i)\in\mathrm{Mor}^0(\msf G,\msf H),\, i=1,2} \mathrm{Mor}^1(\G,\H)((\psi_1,\K_1,\u_1),(\psi_2,\K_2,\u_2))
\]
the set of all arrows between morphisms from $\sf G$ to $\sf H$.
\end{defn}

In the following we will define vertical\footnote{We use ``vertical composition'' to distinguish it from another composition of arrows constructed in the next section. This also fits with the terminology in 2-category/bicategory (cf. \cite{Leinster1998,Meyer-Zhu2015,Shen2014}).} composition of arrows and show that
\[
{\sf Mor(G,H)}=(\mathrm{Mor}^1({\sf G,H})\rrto \mathrm{Mor}^0(\sf G,H))
\]
is a groupoid.

Given two arrows $\alpha_i\in \mathrm{Mor^1}(\G,\H) ((\psi_i,\K_i,\u_i),(\psi_{i+1},\K_{i+1},\u_{i+1}))$ for $i=1,2$, the vertical composition
\[
\alpha_1\bullet\alpha_2\in \mathrm{Mor^1}(\G,\H)((\psi_1,\K_1,\u_1),(\psi_3,\K_3,\u_3)),
\]
is constructed as follow.

\begin{construction}\label{C vert-compose}
Set $\K_{12}:=\K_1\times_\G\K_2$, $\K_{23}:=\K_2\times_\G\K_3$, $\K_{13}:=\K_1\times_\G\K_3$, $\K_{12,23}:=\K_{12}\times_{\K_2}\K_{23}$. We have the following diagram
\begin{align}\label{E dig-D-vert-cpos}
\bfig
    \cube|blrb|/@{->}^<>(.6){\Phi}`->`->`>/
    <1500,900>[\K_{12,23}`\K_{13}`\K_{23}` \K_3; ```\hskip 1cm]
    (700,300)|alrb|<1400,800>[\K_{12}`\K_1`\K_2`\H;```\hskip 1cm]
    |lrrr|/->`->`->`->/[```]
    \place(1000,800)[\twoar(0,-1)]
    \place(1100,770)[\alpha_1]
    \place(1000,200)[\twoar(-12,-10)]
    \place(1110,170)[\alpha_2]
    \place(1650,700)[\twoar(0,-1)]
    \place(1900,670)[\alpha_1\bullet\alpha_2?]
    \efig
\end{align}
in which all unmarked strict morphisms are natural projections from fiber products to their factors. Objects and arrows in $\K_{12,23}$ are of the form
\begin{align}\label{E arrow-W1223}
\begin{split}
\xymatrix{
k_1\ar[d]^a  \ar@{.>}[r]^x \ar@{.>}[r]_{\G}&
k_2\ar[d]^b  \ar[r]^z \ar@{.>}[r]_{\K_2}&
k_2'\ar[d]^c  \ar@{.>}[r]^y \ar@{.>}[r]_{\G}&
k_3\ar[d]^d  \\
\tilde k_1\ar@{.>}[r]^{\tilde x} \ar@{.>}[r]_{\sf G}&
\tilde k_2\ar[r]^{\tilde z} \ar@{.>}[r]_{\K_2}&
\tilde k_2'\ar@{.>}[r]^{\tilde y} \ar@{.>}[r]_{\sf G} &
\tilde k_3,}
\end{split}\end{align}
with two rows being two objects in $K^0_{12,23}$ and all four columns combine into an arrow in $K^1_{12,23}$, where
\begin{itemize}
\item $k_1,\tilde k_1\in  K_1^0$, $k_2,k_2',\tilde k_2,\tilde k_2'\in  K_2^0$, $k_3,\tilde k_3\in  K_3^0$, and

\item $a\in  K_1^1$, $b,c\in  K_2^1$, $d \in  K_3^1$, and

\item $x,y,\tilde x,\tilde y\in G^1$, $z,\tilde z\in  K_2^1$, and
\item $\psi_1^1(a)\cdot \tilde x=x\cdot \psi_2^1(b)$, $b\cdot\tilde z=z\cdot c$, $\psi_2^1(c)\cdot\tilde y=y\cdot\psi^1_3(d)$.
\end{itemize}
The strict morphism $\Phi:\K_{12,23} \rto  \K_{13}$ is given by
\begin{align}\label{eq Psi}
\begin{split}
\Phi^0(k_1,x,k_2,z,k_2', y,k_3) &= (k_1, x\cdot \phi_2^1(z)\cdot y,k_3),\\
\Phi^1(a,b,(k_1,x,k_2,z,k_2',y,k_3),c,d)&= (a,(k_1, x\cdot \phi_2^1(z)\cdot y,k_3),d).
\end{split}
\end{align}
In the cube \eqref{E dig-D-vert-cpos}, the square with vertices $\{\K_{12,23},\K_{12},\K_{23}, \K_2\}$ has a natural transformation between the two composed strict morphisms from $\K_{12,23}$ to $\K_2$. By the definition of $\Phi$ and projections, the two squares with vertices $\{\K_{12,23},\K_{12},\K_{13}, \K_1\}$ and $\{\K_{12,23},\K_{23},\K_{13}, \K_3\}$ are commutative. Therefore five faces of the cube \eqref{E dig-D-vert-cpos} have natural transformations except the face on the very right, which is the $\alpha_1\bullet\alpha_2$ that we will define.

The vertical composition $\alpha_1\bullet\alpha_2:K_{13}^0\rto H^1$ is given by
\begin{align}\label{E alpha-odot-beta}
\alpha_1\bullet\alpha_2(k_1,x,k_3) =\alpha_1(k_1,x_1,k_2)\cdot\alpha_2(k_2,x_2,k_3)
\end{align}
for some splitting of $x$ into $\psi_1^0(k_1)\xrightarrow{x_1} \psi_2^0(k_2)\xrightarrow{x_2} \psi_3^0(k_3)$ in $G^1$ with $k_2\in K_2^0$ and $x=x_1\cdot x_2$. Therefore $(k_1,x_1,k_2,1_{k_2},k_2,x_2,k_3)\in  K_{12,23}^0$ and satisfies $\Phi^0(k_1,x_1,k_2,1_{k_2},k_2,x_2,k_3)=(k_1,x,k_3)$. It is direct to verify that this definition does not depend on the choices of the splitting of $x$ and $(\psi_1,\K_1,\u_1)
\xrightarrow{\alpha_1\bullet\alpha_2} (\psi_3,\K_3,\u_3)$.
\end{construction}

\begin{lem}\label{L vert-comp-asso}
The vertical composition of arrows is associative.
\end{lem}

\begin{proof}
Take three arrows $\alpha_i\in \mathrm{Mor^1}(\G,\H)((\psi_i,\K_i,\u_i), (\psi_{i+1},\K_{i+1},\u_{i+1}))$ for $i=1,2,3$. First of all $\alpha_1\bullet\alpha_2:({\K_1\times_\G \K_3})^0\rto H^1$ is given by $\alpha_1\bullet\alpha_2(k_1,x,k_3)= \alpha_1(k_1,x_1,k_2)\cdot\alpha_2(k_2,x_2,k_3)$ with $x=x_1\cdot x_2$. Then $(\alpha_1\bullet\alpha_2)\bullet\alpha_3: ({\K_1\times_\G\K_4})^0\rto H^1$ is given by
\begin{align*}
(\alpha_1\bullet\alpha_2)\bullet\alpha_3(k_1,x,k_4)
&=(\alpha_1\bullet\alpha_2)(k_1,x_1,k_3) \cdot \alpha_3(k_3,x_2,k_4)\\
&=\alpha_1(k_1,x_{11},k_2)\cdot\alpha_2(k_2,x_{12},k_3)
\cdot \alpha_3(k_3,x_2,k_4),
\end{align*}
with $x=x_{11}\cdot x_{12}\cdot x_2$ and $x_1=x_{11}\cdot x_{12}$.

Similarly $\alpha_1\bullet(\alpha_2\bullet\alpha_3): ({\K_1\times_\G\K_4})^0\rto H^1$ is given by
\begin{align*}
\alpha_1\bullet(\alpha_2\bullet\alpha_3)(k_1,x,k_4)
&=\alpha_1(k_1,\tilde x_1,\tilde k_2) \cdot
\alpha_2\bullet\alpha_3(\tilde k_2,\tilde x_2,k_4)\\
&=\alpha_1(k_1,\tilde x_1,\tilde k_2)\cdot
\alpha_2(\tilde k_2,\tilde x_{21},k_3) \cdot
\alpha_3(\tilde k_3,\tilde x_{22},k_4).
\end{align*}
with $x=\tilde x_1\cdot\tilde x_{21}\cdot\tilde x_{22}$ and $\tilde x_2=\tilde x_{21}\cdot \tilde x_{22}$. We could take $k_i=\tilde k_i$ for $i=2,3$, and $x_{11}=\tilde x_1, x_{12}=\tilde x_{21},x_2=\tilde x_{22}$. Therefore $(\alpha_1\bullet\alpha_2)\bullet\alpha_3 =\alpha_1\bullet(\alpha_2\bullet\alpha_3)$.
\end{proof}

There are also unit arrows with respect to vertical composition.
\begin{lem}\label{L unit-arrow}
Given a morphism $(\psi,\K,\u)\in\mathrm{Mor}^0(\sf G,H)$, there is an arrow $\sf 1_{\sf (\psi,K,u)}$ serves as the unit arrow over $\sf (\psi,K,u)$ with respect to the vertical composition $\bullet$ in $\mathrm{Mor}^1(\sf G,H)$.
\end{lem}
\begin{proof}
$\sf 1_{(\psi,K,u)}$ is a natural transformation $\u\circ\pi_1\Rto \u\circ \pi_2$, which as a map ${\sf 1_{(\psi,K,u)}}:({\sf K\times_G K})^0\rto H^1$ is given by
\begin{align}\label{E def-1u}
{\sf 1}_{\sf (\psi,K,u)}(k,x,k'):= u^1((\psi^1)\inv(x)).
\end{align}
\end{proof}

The inverse arrow of an arrow also exists.

\begin{lem}\label{L inv-arrow}
Given an arrow $\alpha\in\mathrm{Mor^1}(\G,\H)((\psi_1,\K_1,\u_1),(\psi_2,\K_2, \u_2))$, there is a natural induced arrow $\alpha\inv\in \mathrm{Mor^1}(\G,\H)((\psi_2,\K_2, \u_2),(\psi_1,\K_1,\u_1))$ satisfying
\[
\alpha\bullet\alpha\inv={\sf 1}_{(\psi_1,\K_1,\u_1)}, \qq\mbox{and}\qq
\alpha\inv\bullet\alpha={\sf 1}_{(\psi_2,\K_2,\u_2)}.
\]
We call $\alpha\inv$ the inverse arrow of $\alpha$ with respective to the vertical composition $\bullet$.
\end{lem}

\begin{proof}
By definition $\alpha$ is a natural transformation $\u_1\circ \pi_1\stackrel{\alpha}{\Rto} \u_2\circ \pi_2:\K_1\times_\G \K_2\rto \H$. Then we see that $\u_2\circ \pi_1\stackrel{\alpha\inv}{\Rto} \u_1\circ \pi_2: \K_2\times_\G \K_1\rto \H$ is $\alpha\inv(k_2,x,k_1):=\alpha(k_1,x\inv,k_2)\inv$.
\end{proof}

Combining Lemma \ref{L vert-comp-asso}, Lemma \ref{L unit-arrow} and Lemma \ref{L inv-arrow} we get

\begin{theorem}\label{T Mor-groupoid}
For each pair $(\G,\H)$ of groupoids,
$\sf Mor(G,H)$ is a groupoid.
\end{theorem}

\subsection{Morphism groupoids via strict fiber products}
\label{Subs MgVsfp}

Now we modify the construction in previous subsection via replacing all fiber products by strict fiber products to construct another morphism groupoid for each pair of groupoids.

\begin{defn}\label{D full-mor}
We call a morphism $(\psi,\sf K,u):\sf G\rhpu H$ a {\em full-morphism} if $\psi$ is a full-equivalence.
\end{defn}

We denote the set of full-morphisms from $\sf G$ to $\sf H$ by $\mathrm{FMor}^0(\sf G,H)$. Hence $\mathrm{FMor}^0(\sf G,H)\subseteq \mathrm{Mor}^0(G,H)$. We could restrict the groupoid $\sf Mor(G,H)$ to $\mathrm{FMor}^0(\sf G,H)$ to get a groupoid. Instead, we use strict fiber products to define arrows between full-morphisms to get a new groupoid.

\begin{defn}\label{D ful-pre-arrow}
For any two full-morphisms $(\psi_1, \K_1,\u_1):G\rhpu H$ and $(\psi_2, \K_2,\u_2):G\rhpu H$, an {\em arrow} $(\psi_1,\K_1,\u_1)\stackrel{\alpha}{\rto} (\psi_2,\K_2,\u_2)$ is a natural transformation $\alpha$ from the strict morphism $\u_1\circ\tilde \pi_1$ to the strict morphism $\u_2\circ\tilde \pi_2$ in the following diagram
\[
\begin{tikzpicture}
\def \x{3}
\def \y{1}
\node (A00) at (0,0)       {$\sf G$};
\node (A10) at (\x,0)      {$\K_1\ctimes_\G \K_2$};
\node (A11) at (\x,\y)     {$\K_1$};
\node (A1-1) at (\x,-1*\y) {$\K_2$};
\node (A30) at (3*\x,0)    {$\sf H$};
\node at (1.8*\x,0) {$\Downarrow \alpha$};
\path (A00) edge [<-] node [auto] {$\scriptstyle{\psi_1}$} (A11);
\path (A00) edge [<-] node [auto,swap] {$\scriptstyle{\psi_2}$} (A1-1);
\path (A11) edge [->] node [auto] {$\scriptstyle{\u_1}$} (A30);
\path (A1-1) edge [->] node [auto,swap] {$\scriptstyle{\u_2}$} (A30);
\path (A10) edge [->] node [auto] {$\scriptstyle{\tilde \pi_1}$} (A11);
\path (A10) edge [->] node [auto,swap] {$\scriptstyle{\tilde \pi_2}$} (A1-1);
\end{tikzpicture}
\]
where $\tilde \pi_i:\K_1\ctimes_\G\K_2\rto  \K_i, i=1,2,$ are the projections.

Denote by $\mathrm{FMor}^1(\G,\H)((\psi_1,\K_1,\u_1),(\psi_2,\K_2,\u_2))$ the set of arrows from $(\psi_1,\K_1,\u_1)$ to $(\psi_2,\K_2,\u_2)$, and set
\[
\mathrm{FMor}^1(\G,\H):=
\bigsqcup_{(\psi_i, \K_i,\u_i)\in\mathrm{FMor}^0(\G,\H),\, i=1,2}
\mathrm{FMor}^1(\G,\H)((\psi_1,\K_1,\u_1),(\psi_2,\K_2,\u_2))
\]
\end{defn}

Given two arrows between full-morphisms $\alpha_i\in\mathrm{FMor}^1(\G,\H)((\psi_i, \K_i,\u_i), (\psi_{i+1}, \K_{i+1},\u_{i+1}))$, $i=1,2$, the vertical composition
\[
\alpha_1\,\tilde\bullet\,\alpha_2\in
\mathrm{FMor}^1(\G,\H)((\psi_1,\K_1,\u_1),(\psi_3,\K_3,\u_3))
\]
is constructed as follow.

\begin{construction}\label{C vert-compose-full}
Set $\tilde\K_{12}:=\K_1\ctimes_\G\K_2$, $\tilde\K_{23}:=\K_2\ctimes_\G\K_3$, $\tilde\K_{13}:=\K_1\ctimes_\G\K_3$, $\tilde\K_{12,23}:=\K_{12}\ctimes_{\K_2} \K_{23}$. We have the following diagram
\begin{align}\label{E dig-D-vert-cpos-full}
\bfig
    \cube|blrb|/@{->}^<>(.6){\Psi}`->`->`>/
    <1500,900>[\tilde\K_{12,23}`\tilde\K_{13}`\tilde\K_{23}` \K_3; ```\hskip 1cm]
    (700,300)|alrb|<1400,800>[\tilde\K_{12}`\K_1`\K_2`\H;```\hskip 1cm]
    |lrrr|/->`->`->`->/[```]
    \place(1000,800)[\twoar(0,-1)]
    \place(1100,760)[\alpha_1]
    \place(1000,200)[\twoar(-12,-10)]
    \place(1110,170)[\alpha_2]
    \place(1650,700)[\twoar(0,-1)]
    \place(1900,670)[\alpha_1\,\tilde\bullet\,\alpha_2?]
    \efig
\end{align}
with unmarked strict morphisms being natural projections from strict fiber products to their factors. An arrow of $\tilde\K_{12,23}$, denoted by $(a,b,(k_1,k_2, k_2,k_3), b,c)$, can be illustrated in the following form
\begin{align}\label{E arrow-K1223-full}
\begin{split}
\xymatrix{
k_1\ar[d]^a  \ar@{.>}[r]^{1_{\psi^0_1(k_1)}}&
k_2\ar[d]^b  \ar[r]^{1_{k_2}}&
k_2\ar[d]^b  \ar@{.>}[r]^{1_{\psi^0_3(k_3)}} &
k_3\ar[d]^c  \\
\tilde k_1\ar@{.>}[r]^{1_{\psi^0_1(\tilde k_1)}}&
\tilde k_2\ar[r]^{1_{\tilde k_2}}&
\tilde k_2\ar@{.>}[r]^{1_{\psi_2^0(\tilde k_2)}} & \tilde k_3}
\end{split}\end{align}
with two rows being two objects in $\tilde K^0_{12,23}$ and all four columns combine into an arrow in $\tilde K^1_{12,23}$, where
\begin{itemize}
\item $k_i,\tilde k_i\in  K_i^0$ being objects of $\K_i$ for $i=1,2,3$,
\item $a\in  K_1^1,b\in  K_2^1,c\in  K_3^1$ being arrows of $\K_i$,  for $i=1,2,3$,
\item $\psi_1^1(a)=\psi_2^1(b)=\psi_3^1(c)$.
\end{itemize}
The strict morphism $\Psi:\tilde \K_{12,23} \rto \tilde \K_{13}$ is given by
\begin{align}\label{eq Psi-full}
\begin{split}
\Psi^0(k_1,k_2, k_2, k_3) &=  (k_1,k_3),\\
\Psi^1(a,b,(k_1,k_2, k_2,k_3), b,c)&=(a,c).
\end{split}
\end{align}
Then from the definition of $\Psi$ and natural projections we see that in the cube \eqref{E dig-D-vert-cpos-full} the three squares with vertices $\{\tilde \K_{12,23},\tilde \K_{12},\tilde \K_{13}, \K_1\}$, $\{\tilde \K_{12,23},\tilde \K_{23},\tilde \K_{13}, \K_3\}$ and $\{\tilde \K_{12,23},\tilde \K_{12},\tilde \K_{23}, \K_2\}$ are all commutative.

The composition $\alpha\,\tilde\bullet\,\beta:\tilde K_{13}^0\rto G^1$ is give by
\begin{align}\label{E alpha-odot-beta-full}
\alpha\,\tilde\bullet\,\beta(k_1,k_3)=\alpha(k_1,k_2)\cdot\beta(k_2,k_3)
\end{align}
for a $k_2\in  K_2^0$ satisfying $\psi_1^0(k_1)=\psi_2^0(k_2)=\psi_3^0(k_3)$. It is direct to verify that $\alpha\,\tilde\bullet\,\beta$ is a natural transformation $\alpha\,\tilde\bullet\,\beta:\u_1\circ\tilde\pi_1\Rto \u_3\circ\tilde \pi_3$, hence an arrow $\alpha\,\tilde\bullet\,\beta:\u_1\rto \u_3$.
\end{construction}

Similar as Lemma \ref{L vert-comp-asso}, Lemma \ref{L unit-arrow} and Lemma \ref{L inv-arrow} we have

\begin{lem}\label{L vert-comp-full-asso}
The vertical compositions
\[
\tilde\bullet\,:\mathrm{FMor^1}(\sf G,H)\times
            \mathrm{FMor^1}(G,H) \rto
            \mathrm{FMor^1}(G,H)
\]
is associative.
\end{lem}

\begin{lem}\label{L unit-ful-arrow}
Given a full-morphism $\sf (\psi,K,u)\in\mathrm{FMor}^0(G,H)$, there is an arrow $\sf 1_{\sf (\psi,K,u)}$ serves as the unit arrow over $\sf (\psi,K,u)$ in $\mathrm{FMor}^1(\sf G,H)$ with respect to the vertical composition $\tilde\bullet$, which is given by
\[
{\sf 1}_{\sf (\psi,K,u)}: ({\sf K\ctimes_G K})^0\rto H^1, \qq
\msf 1_{\sf (\psi,K,u)}(k,k'):= u^1((\psi^1)\inv(1_{\psi^0(k)})).
\]
\end{lem}

\begin{lem}\label{L inv-full-arrow}
Given an arrow $\alpha\in\mathrm{FMor}^1(\G,\H)((\psi_1,\K_1,\u_1), (\psi_2,\K_2,\u_2))$ there is a natural induced arrow $\alpha\inv\in\mathrm{FMor}^1(\G,\H) ((\psi_2,\K_2,\u_2),(\psi_1,\K_1,\u_1))$, which is given by
\[
\alpha\inv:({\K_2\ctimes_\G \K_1})^0\rto H^1,\qq
\alpha\inv(k_2,k_1):=\alpha(k_1,k_2)\inv.
\]
It satisfies
\[
\alpha\,\tilde\bullet\,\alpha\inv=\msf 1_{(\psi_1,\K_1,\u_1)}, \qq\mbox{and}\qq
\alpha\inv\,\tilde\bullet\,\alpha=\msf 1_{(\psi_2,\K_2,\u_2)}.
\]
\end{lem}

Therefore

\begin{theorem}\label{T FMor-groupoid}
For each pair $(\G,\H)$ of groupoids, $\sf FMor(G,H)=(\mathrm{FMor}^1(G,H)\rrto \mathrm{FMor^0}(G,H))$ is a groupoid.
\end{theorem}

\subsection{Equivalence between $\sf Mor(G,H)$ and $\sf FMor(G,H)$}
\label{Subs equi-of-mor-groupoid}

We have the following equivalence of morphism groupoids.

\begin{theorem}\label{T equi-i}
There is a natural strict morphisms $\i=(i^0,i^1):\sf FMor(G,H)\rto Mor(G,H)$. Moreover it is an equivalence between groupoids.
\end{theorem}
\begin{proof}
We first construct the $\sf i$. The $i^0:\sf \mathrm{FMor}^0(G,H)\rto  \mathrm{Mor}^0(G,H)$ is the inclusion.

We next define $i^1$. Take an arrow $\alpha\in\mathrm{FMor^1}(\G,\H)(\sf(\psi,K,u),(\phi,L,v))$, then $\alpha$ is a natural transformation
\[
\u\circ\tilde \pi_1\stackrel{\alpha}{\Rto} \v\circ \tilde \pi_2
:\sf K\ctimes_GL\rto H.
\]
By Lemma \ref{L Go*/HK-eq-G*/HK}, $\sf q:K\ctimes_GL\rto K\times_GL$, is an injective equivalence. This $\sf q$ together with the equalities $\u\circ \tilde \pi_1=\u\circ  \pi_1\circ \q$, $\v\circ \tilde \pi_2=\v\circ  \pi_2\circ \q$, gives rise to a canonically induced natural transformation
\[
\u\circ  \pi_1\stackrel{\tilde\alpha}{\Rto} \v\circ  \pi_2:
\sf K\times_{\psi,G,\phi}L \rto H
\]
described as follow. Since $\sf q$ is an equivalence,  for any object $b\in \sf (K\times_GL)^0$, there is an object $a\in(\sf K\ctimes_GL)^0$ and an arrow $x:q^0(a)\rto b$ in $(\sf K\times_GL)^1$. Then we set
\begin{align}\label{E extend-alpha-S5}
\tilde\alpha(b):=
[(\msf u\circ\pi_1)^1(x)]\inv\cdot\alpha(a)\cdot
(\msf v\circ\pi_2)^1(x).
\end{align}
By using the fact that $\alpha$ is a natural transformation, it is direct to verify that this definition of $\tilde\alpha$ does not depend on the choices of $a$ and $x$ and is a natural transformation. Then we set $i^1(\alpha)=\tilde \alpha$.

We next show that $i=(i^0,i^1)$ is a strict morphism and an equivalence. By the construction above $i^0{\sf(\psi,K,u)}\xrightarrow{i^1(\alpha)} i^0\sf(\phi,L,v)$. We next show that it also preserves the vertical composition. For two arrows $\sf (\psi,K,u)\xrightarrow{\alpha} (\phi,L,v)$, $\sf (\psi,L,v)\xrightarrow{\beta} (\phi,M,w)$ in $\mathrm{FMor}^1(\sf G,H)$ we have $\alpha\,\tilde\bullet\,\beta:\sf (K\ctimes_GM)^0\rto H^1$, with
\[
\alpha\,\tilde\bullet\,\beta(k,m)=\alpha(k,l)\cdot\beta(l,m)
\]
for some $l\in L^0$ satisfying $\phi^0(l)=\psi^0(k)=\varphi^0(m)$. We next show that $i^1(\alpha\,\tilde\bullet\,\beta)=i^1(\alpha)\bullet i^1(\beta): ({\sf K\times_G M})^0\rto H^1$.

Objects in $(\sf K\times_GM)^0$ is of the form $(k,x,m)$ with $x:\psi^0(k)\rto \varphi^0(m)$ in $ G^1$. Take an arrow in $({\sf K\times_GM})^1$
\[
(1_k,(\varphi^1)\inv(x)): q^0(k,m')\rto (k,x,m),
\]
where $m'$ satisfies $\phi^0(m')=\psi^0(k)$ (See similar construction in the proof of Lemma \ref{L Go*/HK-eq-G*/HK}). Then
\begin{align*}
i^1(\alpha\,\tilde{\bullet}\,\beta)(k,x,m)
&=[(\u\circ\pi_1)^1(1_k,(\varphi^1)\inv(x))]\inv
\cdot\alpha\,\tilde{\bullet}\,\beta(k,m')\cdot
( \w\circ\pi_2)^1(1_k,(\varphi^1)\inv(x))\\
&=u^1(1_k)\cdot
\alpha\,\tilde{\bullet}\,\beta(k,m')\cdot
w^1((\varphi^1)\inv(x))\\
&=\alpha(k,l')\cdot\beta(l',m')\cdot
w^1((\varphi^1)\inv(x))
\end{align*}
for some $l'\in L^0$ such that $\psi^0(k)=\phi^0(l')=\varphi^0(m')$.

On the other hand, for this $l'$, we have $(k,1_{\psi^0(k)},l')\in \mbox{Im}\, q^0$ and $(1_{l'},(\varphi^1)\inv(x)):(l',1_{\psi^0(k)},m')\rto (l',x,m)$ is also an arrow in ${\sf (L\times_GM)}^1$. Therefore
\begin{align*}
& i^1(\alpha)\bullet i^1(\beta) (k,x,m)\\
&= i^1(\alpha)(k,1_{\psi^0(k)},l')\cdot i^1(\beta)(l',x,m)\\
&= i^1(\alpha)(k,1_{\psi^0(k)},l')\cdot
[(\v\circ\pi_1)^1(1_{l'},(\varphi^1)\inv(x))]\inv
\cdot\beta(l',m')\cdot
(\msf w\circ\pi_2)^1(1_{l'},(\varphi^1)\inv(x))\\
&=\alpha(k,l')\cdot v^1(1_{l'})
\cdot\beta(l',m')\cdot
w^1((\varphi^1)\inv(x))\\
&=\alpha(k,l')\cdot\beta(l',m')\cdot
w^1((\varphi^1)\inv(x)).
\end{align*}
Hence $i^1(\alpha)\bullet i^1(\beta)=i^1(\alpha\,\tilde{\bullet}\,\beta)$. Consequently $\i=(i^0,i^1):\sf FMor(G,H)\rto Mor(G,H)$ is a strict morphism.

We next show that $\sf i$ is an equivalence. First of all, for every morphism $\sf(\psi, K, u)\in\mathrm{Mor^0(G,H)}$ there is another morphism $({\sf id_G}\circ\pi_1,{\sf G\times_GK,u}\circ\pi_2) \in\mathrm{Mor^0(G,H)}$. Obviously $({\sf id_G\circ}\pi_1,{\sf G\times_GK,u}\circ\pi_2)$ is a full-morphism, hence belongs to $\text{Im}\, i^0$. We claim that there is an arrow $\alpha\in \mathrm{Mor}^1(\G,\H)({\sf (\psi, K, u)}, ({\sf id_G}\circ\pi_1,{\sf G\times_GK,u}\circ\pi_2))$. Now we construct the $\alpha$. By definition $\alpha$ is a natural transformation $\msf u\circ\pi_1\stackrel{\alpha}{\Rto} \u\circ\pi_2\circ \pi_2: \sf K\times_G(G\times_GK)\rto H$. Set $\Q=(Q^1\rrto Q^0):=\sf K\times_G(G\times_GK)$. An object in $Q^0$ is of the form
\[
\xymatrix{k\ar@{.>}[r]^x_{\sf G} & (g \ar@{.>}[r]^{y}_{\sf G} & k')}
\]
denoted by $(k,x,g,y,k')$. It is mapped by $\u\circ\pi_1$ and $\u\circ\pi_2\circ \pi_2$ respectively to $u^0(k)$, $u^0(k')$. From $(k,x,g,y,k')$ we get an arrow $x\cdot y:\psi^0(k)\rto \psi^0(k')$ in $G^1$. Then since $\psi$ is an equivalence we get a unique arrow $(\psi^1)\inv(x\cdot y):k\rto k'$ in $K^1$. We define
\[
\alpha(k,x,g,y,k'):=u^1((\psi^1)\inv(x\cdot y))=
u^1\circ (\psi^1)\inv(x)\cdot u^1\circ (\psi^1)\inv(y).
\]
Then it is direct to check that this is the arrow we want. Hence $i^0$ is essentially surjective.

Finally we show that $\sf i$ is full and faithful. In fact the inverse map $(i^1)\inv$ of
\[
i^1:\mathrm{FMor}^1(\G,\H)((\psi_1,\K_1,\u_1),(\psi_2,\K_2,\u_2)) \rto \mathrm{Mor^1}(\G,\H)(i^0(\psi_1,\K_1,\u_1),i^0( \psi_2,\K_2,\u_2))
\]
is given by
\[
(i^1)\inv(\alpha):=\alpha\circ q^0:
({\K_1\ctimes_\G \K_2})^0\rto({\K_1\times_\G \K_2})^0 \rto H^1.
\]
Hence $\sf i$ is an equivalence.
\end{proof}

This $\i:\sf FMor(\G,\H)\rto Mor(\G,\H)$ factors through
\[
\i: {\sf FMor(G,H)}\rto {\sf Mor(G,H)}|_{\mathrm{FMor^0}(\sf G,H)}
\hrto \sf Mor(G,H).
\]
Hence all three groupoids are equivalent.

\begin{rem}\label{R bullet-to-tilde-bullet}
We can give another definition of $\tilde\bullet$ via the construction of $i^1$ and definition of $\bullet$. Suppose we have two arrows in $\alpha,\beta \in \mathrm{FMor}^1(\G,\H)(\sf (\psi,K,u),(\phi,L,v))$, then via $i^1$ we get two arrows
\[
i^1(\alpha),i^1(\beta) \in \mathrm{Mor}^1(\G,\H)(\sf (\psi,K,u),(\phi,L,v)).
\]
Then we have $i^1(\alpha)\bullet i^1(\beta)$, and
\[
\alpha\,\tilde\bullet\,\beta=(i^1)\inv(i^1(\alpha)\bullet i^1(\beta))
=(i^1(\alpha)\bullet i^1(\beta))\circ q^0,
\]
with $q^0:\tilde K_{13,23}^0\rto K_{13,23}^0$.

In fact, the injective strict morphism $\q$ in \eqref{E q-geo-fbi-to-fib} from strict fiber product to fiber product together with identity strict morphisms of $\K_1,\K_2,\K_3,H$ gives us a strict morphisms from the cube \eqref{E dig-D-vert-cpos-full} to the cube \eqref{E dig-D-vert-cpos}, and $\Psi$ is the composition of $\Phi$ and $q^0:\tilde
K_{13,23}^0\rto K_{13,23}^0$.
\end{rem}

\section{Composition of morphism groupoids}
\label{S Composition-Mor-FMor}

In this section we show that there are natural composition functors on morphism groupoids:
\begin{align*}
&\circ:\sf Mor(G,H)\times Mor(H,N)\rto  Mor(G,N),\\
\mbox{and}\qq&\tcirc:\sf FMor(G,H)\times FMor(H,N)\rto  FMor(G,N).
\end{align*}

\subsection{Composition functor ``$\circ$''}

Given two morphisms $\sf (\psi, K,u): G\rhpu H \mbox{ and }
\sf (\phi, L, v): H\rhpu N$, let $\sf M:=K\times_HL$ be the fiber product of $\sf u:K\rto H$ and $\phi:\sf L\rto H$, let $\pi_1:\sf M\rto K$ and $\pi_2: \sf M\rto L$ be the corresponding projections. By Lemma \ref{L G*/HK-equiv G}, $\pi_1$ is an equivalence. Hence by Lemma \ref{L eqs-compose-eq}, $\psi\circ\pi_1:\sf M\rto G$ is also an equivalence.

\begin{defn}\label{D copose-morphism}
The {\em composition} of $\sf (\psi, K,u)$ and $\sf (\phi, L, v)$ is defined to be
\[
{\sf (\phi, L, v)\circ (\psi, K,u)}
:=(\psi\circ \pi_1, \M, \v\circ \pi_2): \sf G\rhpu N.
\]
This can be summarized in the following diagram
\[
\xymatrix{&& \M\ar[dl]_{\pi_1}\ar[dr]^{\pi_2}&&\\\sf G &\sf K \ar[l]_\psi\ar[r]^{\sf u} & \sf H&\sf L\ar[l]_\phi\ar[r]^{\sf v} & \sf N.}
\]
\end{defn}

For a groupoid $\sf G$, we call \[\sf 1_G:=(id_G,G,id_G):\xymatrix{\G &\G\ar[l]_-{\sf id_G} \ar[r]^-{\sf id_G}& \G}\] the {\em identity morphism} of $\sf G$. We also denote it by $\sf 1_G$. However it is not the unit for composition of morphisms since $\sf G\times_G H\not= H$.

Now we describe the horizontal composition of arrows. Take two arrows
\[
(\psi_1,\K_1,\u_1)\xrightarrow{\alpha}(\psi_2,\K_2,\u_2):\G\rhpu \H,\qq
(\phi_1,\J_1,\v_1)\xrightarrow{\beta}(\phi_2,\J_2,\v_2):\sf H\rhpu N.
\]
The horizontal composition $\beta\circ \alpha$ of $\alpha$ and $\beta$ is an arrow
\[
(\phi_1,\J_1,\v_1)\circ(\psi_1,\K_1,\u_1)
\xrightarrow{\beta\circ \alpha} (\phi_2,\J_2,\v_2)\circ (\psi_2,\K_2,\u_2):\G\rhpu \N.
\]
We describe the construction.

\begin{construction}\label{C hori-compose}
Set $\K_{12}:=\K_1\times_\G \K_2$, $\L:=\K_1\times_\H\J_1$, $\J_{12}:=\J_1\times_\H\J_2$, $\M:=\K_2\times_\H\J_2$, and $\U:=\sf L\times_G M$. We have the following diagram
\begin{align}\label{E dig-hori-cpose}\begin{split}
\begin{tikzpicture}
\def \x{2.5}
\def \y{0.8}
\node (A00)   at (0,0)        {$\sf G$};
~
\node (A11)   at (\x,\y)      {$\K_1$};
\node (A10)   at (\x,0)       {$\K_{12}$};
\node (A1-1)  at (\x,-1*\y)   {$\K_2$};
~
\node (A21)   at (2*\x,\y)    {$\sf L$};
\node (A20)   at (2*\x,0)     {$\sf H$};
\node (A2-1)  at (2*\x,-1*\y) {$\sf M$};
~
\node (A31)   at (3*\x,\y)    {$\J_1$};
\node (A30)   at (3*\x,0)     {$\J_{12}$};
\node (A3-1)  at (3*\x,-1*\y) {$\J_2$};
~
\node (A40)   at (4*\x,0)     {$\sf N$};
~
\node at (1.5*\x,0) {$\Downarrow\alpha$};
\node at (3.5*\x,0) {$\Downarrow\beta$};
~
\node (A-10) at (-0.3*\x,0) {$\sf U$};
~
\path (A00) edge [
<-]  node [auto] {$\scriptstyle{\psi_1}$} (A11);
\path (A00) edge [
<-]  node [auto,swap] {$\scriptstyle{\psi_2}$} (A1-1);
~
\path (A10) edge [->]  node [auto] {$\scriptstyle{\pi_1}$} (A11);
\path (A10) edge [->]  node [auto,swap] {$\scriptstyle{\pi_2}$} (A1-1);
~
\path (A11) edge [->]  node [auto] {$\scriptstyle{\u_1}$} (A20);
\path (A1-1) edge [->]  node [auto,swap] {$\scriptstyle{\u_2}$} (A20);
~
\path (A11) edge [
<-]  node [auto] {$\scriptstyle{\pi_1}$} (A21);
\path (A1-1) edge [
<-]  node [auto,swap] {$\scriptstyle{\pi_1}$} (A2-1);
~
\path (A21) edge [
->]  node [auto] {$\scriptstyle{\pi_2}$} (A31);
\path (A2-1) edge [
->]  node [auto,swap] {$\scriptstyle{\pi_2}$} (A3-1);
~
\path (A20) edge [<-]  node [auto] {$\scriptstyle{\phi_1}$} (A31);
\path (A20) edge [<-]  node [auto,swap] {$\scriptstyle{\phi_2}$} (A3-1);
~
\path (A30) edge [->]  node [auto] {$\scriptstyle{\pi_1}$} (A31);
\path (A30) edge [->]  node [auto,swap] {$\scriptstyle{\pi_2}$} (A3-1);
~
\path (A40) edge [
<-]  node [auto,swap] {$\scriptstyle{\v_1}$} (A31);
\path (A40) edge [
<-]  node [auto] {$\scriptstyle{\v_2}$} (A3-1);
~
\path (A-10) edge [
->,bend left=25]node[auto]{$\scriptstyle{\pi_1}$} (A21);
\path (A-10) edge [
->,bend right=25]node[auto,swap] {$\scriptstyle{\pi_2}$} (A2-1);
~
\path (A21) edge [
->,bend left=30]
                node[auto]{$\scriptstyle{\v_1\circ\pi_2}$} (A40);
\path (A2-1) edge [
->,bend right=30]
                node[auto,swap] {$\scriptstyle{\v_2\circ\pi_2}$} (A40);
\end{tikzpicture}\end{split}
\end{align}
The arrow $\beta\circ \alpha$ we want is a natural transformation $\beta\circ \alpha:\v_1\circ\pi_2\circ\pi_1\Rto \v_2\circ\pi_2\circ\pi_2:\sf U\rto N$. An objects in $U^0$ is of the form
\begin{align*}
\begin{split}
\xymatrix{
j_1 &
k_1 \ar@{.>}[l]_x \ar@{.>}[l]^{\sf H}
    \ar@{.>}[r]^z \ar@{.>}[r]_{\sf G}
    &
k_2 \ar@{.>}[r]^y \ar@{.>}[r]_{\sf H}
    &
j_2, }
\end{split}
\end{align*}
denoted by $(j_1,x,k_1,z,k_2,y,j_2)$. We define the horizontal composition by
\[
\beta\circ \alpha(j_1,x,k_1,z,k_2,y,j_2):= \beta(j_1,x\inv\cdot\alpha(k_1,z,k_2)\cdot y,j_2).
\]
It is direct to verify that this is an arrow $\beta\circ \alpha:(\phi_1,\J_1,\v_1)\circ(\psi_1,\K_1,\u_1) \rto (\phi_2,\J_2,\v_2)\circ (\psi_2,\K_2,\u_2)$.
\end{construction}

\begin{lem}\label{L hori-functor}
Combine with composition of morphisms we get a horizontal composition functor
\[
\circ:\sf Mor(H,N)\times Mor(G,H)\rto Mor(G,N).
\]
So ``$\circ$'' is a strict morphism of groupoids.
\end{lem}

\begin{proof}
For $i=1,2$, take
\begin{align*}
\alpha_i:(\psi_i,\K_i,\u_i)\rto(\psi_{i+1},\K_{i+1},\u_{i+1})&:\G\rhpu \H\\
\beta_i:(\phi_i,\J_i,\v_i)\rto(\phi_{i+1},\J_{i+1},\v_{i+1})&:\H\rhpu \N.
\end{align*}
Then we need to show that
\[
(\beta_1\circ\alpha_1)\bullet(\beta_2\circ\alpha_2)\stackrel{?}{=}
(\beta_1\bullet\beta_2)\circ(\alpha_1\bullet\alpha_2).
\]
Note that they are both defined on the object space of
\[
\Q=(Q^1\rrto Q^0):=(\K_1\times_\H\J_1)\times_\G
(\K_3\times_\H\J_3).
\]
We first compute $(\beta_1\circ\alpha_1)\bullet(\beta_2\circ\alpha_2)$. Take an object $(j_1,x,k_1,z,k_3,y,j_3)\in Q^0$. By definition
\begin{align*}
&(\beta_1\circ\alpha_1)\bullet(\beta_2\circ\alpha_2)
(j_1,x,k_1,z,k_3,y,j_3)\\
=&\beta_1\circ\alpha_1(j_1,x,k_1,z_1,k_2,w,j_2)
\cdot\beta_2\circ\alpha_2(j_2,w,k_2,z_2,k_3,y,j_3)
\end{align*}
for some $(k_2,w,j_2)\in (\K_2\times_\H\J_2)^0$ and $z=z_1\cdot z_2$ in $G^1$. Then we get
\begin{align*}
&(\beta_1\circ\alpha_1)\bullet(\alpha_2\circ\beta_2) (j_1,x,k_1,z,k_3,y,j_3)\\
=&\beta_1(j_1,x\inv\cdot\alpha_1(k_1,z_1,k_2)\cdot w,j_2)
\cdot
\beta_2(j_2,w\inv\cdot\alpha_2(k_2,z_2,k_3)\cdot y,j_3)
\end{align*}
Similarly
\begin{align*}
&(\beta_1\bullet\beta_2)\circ(\alpha_1\bullet\alpha_2)
(j_1,x,k_1,z,k_3,y,j_3)\\
&=\beta_1\bullet\beta_2
(j_1,x\inv\cdot\alpha_1\bullet\alpha_2(k_1,z,k_3)\cdot y,j_3)\\
&=\beta_1\bullet\beta_2(j_1, x\inv\cdot\alpha_1(k_1,z_1,k_2)
\cdot\alpha_2(k_2,z_2,k_3)\cdot y, j_3)\\
&=\beta_1(j_1,x\inv\cdot\alpha_1(k_1,z_1,k_2)\cdot w,j_2)
\cdot \beta_2(j_2,w\inv\cdot\alpha_1(k_1,z_1,k_2)\cdot y,j_3),
\end{align*}
with $k_2,w,j_2,z_1,z_2$ being same as those above. This finishes the proof.
\end{proof}

\begin{lem}\label{L horizontal-asso}
Under the canonical isomorphism of fiber products in Lemma \ref{L fib-prdct-asso}, the horizontal composition functor ``$\circ$'' is associative.
\end{lem}
\begin{proof}
Take three arrows in $\mathrm{Mor}^1$ as follow:
\[
\begin{tikzpicture}
\def \x{2}
\def \y{0.8}
~
\node (A00) at (0,0)       {$\sf G$};
\node (A10) at (\x,0)      {$\I_{12}$};
\node (A20) at (2*\x,0)    {$\sf H$};
\node (A30) at (3*\x,0)    {$\J_{12}$};
\node (A40) at (4*\x,0)    {$\sf N$};
\node (A50) at (5*\x,0)    {$\K_{12}$};
\node (A60) at (6*\x,0)    {$\sf M$};
~
\node (A01) at (\x,\y)       {$\I_1$};
\node (A11) at (3*\x,\y)     {$\J_1$};
\node (A21) at (5*\x,\y)     {$\K_1$};
~
\node (A0-1) at (\x,-1*\y)       {$\I_2$};
\node (A1-1) at (3*\x,-1*\y)     {$\J_2$};
\node (A2-1) at (5*\x,-1*\y)     {$\K_2$};
~
\path (A00) edge [<-] node [auto] {$\scriptstyle{\psi_1}$} (A01);
\path (A00) edge [<-] node [auto,swap] {$\scriptstyle{\psi_2}$} (A0-1);
~
\path (A10) edge [->] node [auto] {$\scriptstyle{\pi_1}$} (A01);
\path (A10) edge [->] node [auto,swap] {$\scriptstyle{\pi_2}$} (A0-1);
~
\path (A01) edge [->] node [auto] {$\scriptstyle{\u_1}$} (A20);
\path (A0-1) edge [->] node [auto,swap] {$\scriptstyle{\u_2}$} (A20);
~~~~~~
\path (A20) edge [<-] node [auto] {$\scriptstyle{\phi_1}$} (A11);
\path (A20) edge [<-] node [auto,swap] {$\scriptstyle{\phi_2}$} (A1-1);
~
\path (A30) edge [->] node [auto] {$\scriptstyle{\pi_1}$} (A11);
\path (A30) edge [->] node [auto,swap] {$\scriptstyle{\pi_2}$} (A1-1);
~
\path (A11) edge [->] node [auto] {$\scriptstyle{\v_1}$} (A40);
\path (A1-1) edge [->] node [auto,swap] {$\scriptstyle{\v_2}$} (A40);
~~~~~~
\path (A40) edge [<-] node [auto] {$\scriptstyle{\varphi_1}$} (A21);
\path (A40) edge [<-] node [auto,swap] {$\scriptstyle{\varphi_2}$} (A2-1);
~
\path (A50) edge [->] node [auto] {$\scriptstyle{\pi_1}$} (A21);
\path (A50) edge [->] node [auto,swap] {$\scriptstyle{\pi_2}$} (A2-1);
~
\path (A21) edge [->] node [auto] {$\scriptstyle{\w_1}$} (A60);
\path (A2-1) edge [->] node [auto,swap] {$\scriptstyle{\w_2}$} (A60);
~~~~~~~~~
\node at (1.5*\x,0)  {$\Downarrow\alpha_1$};
\node at (3.5*\x,0)  {$\Downarrow\alpha_2$};
\node at (5.5*\x,0)  {$\Downarrow\alpha_3$};
\end{tikzpicture}
\]
with $\I_{12}=\I_1\times_\G \I_2$, $\J_{12}=\J_1\times_\H \J_2$, $\K_{12}=\K_1\times_\N \K_2$. We first consider the compositions of $(\psi_1,\I_1,\u_1)$, $(\phi_1,\J_1,\v_1)$ and $(\varphi_1,\K_1,\w_1)$. We get two compositions
\[
(\psi\circ\pi_1,\I_1\times_\H(\J_1\times_\N\K), \w\circ\pi_2\circ\pi_2)
\qq
\mbox{and}
\qq
(\psi\circ\pi_1\circ\pi_1,(\I_1\times_\H\J_1)\times_\N\K, \w\circ\pi_2).
\]
Then via the canonical isomorphism $\I_1\times_\H(\J_1\times_\N\K)\cong (\I_1\times_\H\J_1)\times_\N\K$ we could identify them. From this natural identification we could get a natural arrow between them. We next show that via such canonical isomorphisms of fiber products, we can also identify $(\alpha_3\circ\alpha_2)\circ\alpha_1$ with $\alpha_3\circ(\alpha_2\circ\alpha_1)$.

The arrow $(\alpha_3\circ\alpha_2)\circ\alpha_1$ is a natural transformation between strict morphisms over
\[
\A:= [(\I_1\times_\H\J_1)\times_\N\K_1]
\times_\G[(\I_2\times_\H\J_2)\times_\N\K_2]
\]
and $\alpha_3\circ(\alpha_2\circ\alpha_1)$ is a natural transformation between strict morphisms over
\[
\B:= [\I_1\times_\H(\J_1\times_\N\K_1)]
\times_\G[\I_2\times_\H(\J_2\times_\N\K_2)].
\]
Under the canonical isomorphisms for fiber products given by Lemma \ref{L fib-prdct-asso} we get a canonical isomorphism $\sf A\cong B$. In particular, the identification over objects is given by
\[\begin{tikzpicture}
\def \x{5}
\node at (0,0) {$\xymatrix{
(i_1\ar@{.>}[r]^x   \ar@{.>}[r]_{\sf H}
   \ar@{.>}[d]_{z} \ar@{.>}[d]^{\sf G} &
j_1)\ar@{.>}[r]^y   \ar@{.>}[r]_{\sf N} &
k_1\\
(i_2\ar@{.>}[r]^{\tilde x}  \ar@{.>}[r]_{\sf H} &
j_2)\ar@{.>}[r]^{\tilde y}  \ar@{.>}[r]_{\sf N} &
k_2,
}$};
\node at (\x,0) {$\xymatrix{
i_1\ar@{.>}[r]^x   \ar@{.>}[r]_{\sf H}
   \ar@{.>}[d]_{z} \ar@{.>}[d]^{\sf G} &
(j_1\ar@{.>}[r]^y   \ar@{.>}[r]_{\sf N} &
k_1)\\
i_2\ar@{.>}[r]^{\tilde x}  \ar@{.>}[r]_{\sf H} &
(j_2\ar@{.>}[r]^{\tilde y}  \ar@{.>}[r]_{\sf N} &
k_2).
}$};
\node at (0.5*\x,0) {$\leftrightarrow$};
\end{tikzpicture}
\]
We write them both as $((i_1,x,j_1,y,k_1),z,(i_2,\tilde x,j_2,\tilde y,k_2))$. Then by definition of horizontal composition of arrows we have
\begin{align*}
(\alpha_3\circ\alpha_2)\circ\alpha_1
((i_1,x,j_1,y,k_1),z,(i_2,\tilde x,j_2,\tilde y,k_2))
&=\alpha_3\circ\alpha_2(j_1,y,k_1,x\inv \cdot
\alpha_1(i_1,z,i_2)\cdot\tilde x,j_2,\tilde y,k_2)\\
&=\alpha_3(k_1,y\inv \cdot
\alpha_2(j_1,x\inv\cdot\alpha_1(i_1,z,i_2)\cdot \tilde x,j_2)\cdot\tilde y,k_2),
\end{align*}
and
\begin{align*}
\alpha_3\circ(\alpha_2\circ\alpha_1)
((i_1,x,j_1,y,k_1),z,(i_2,\tilde x,j_2,\tilde y,k_2))
&=\alpha_3(k_1,y\inv \cdot \alpha_2\circ\alpha_1(i_1,x,j_1,z,i_2,\tilde x,j_2)\cdot\tilde y,k_2)\\
&=\alpha_3(k_1,y\inv \cdot
\alpha_2(j_1,x\inv\cdot\alpha_1(i_1,z,i_2)\cdot \tilde x,j_2)\cdot\tilde y,k_2).
\end{align*}
Hence we could identify the two kinds of compositions of morphisms and arrows simultaneously via the canonical isomorphisms of fiber products. This finishes the proof.
\end{proof}

\subsection{Composition functor ``$\tcirc$''}

\begin{defn} \label{D comp-full-mor}
Given two full-morphisms $\sf (\psi,K,u):G\rhpu H$ and $\sf (\phi,L,v):H\rhpu N$, the {\em composition} is defined to be
\[
{\sf (\psi,K,u)\,\tilde \circ\, (\phi,L,v)}:=
(\psi\circ\tilde \pi_1, {\sf K\ctimes_H L,v}\circ\tilde \pi_2):
\G\rhpu \N,
\]
which is summarized in the following diagram
\[
\xymatrix{&&
\K\ctimes_\H \L\ar[dl]_{\tilde \pi_1}\ar[dr]^{\tilde \pi_2} &&\\
\G &\K \ar[l]_\psi\ar[r]^{\u} & \H &
\L\ar[l]_\phi\ar[r]^{\v} & \N.}
\]
\end{defn}

Now suppose we have two arrows between full-morphisms $(\psi_1,\K_1,\u_1)\xrightarrow{\alpha} (\psi_2,\K_2,\u_2):\sf G\rhpu H$, $(\phi_1,\J_1,\v_1)\xrightarrow{\beta} (\phi_2,\J_2,\v_2): \sf H\rhpu N$. The horizontal composition $\beta\,\tilde\circ\,\alpha$ should be an arrow
\[
(\phi_1,\J_1,\v_1)\,\tilde\circ\,(\psi_1,\K_1,\u_1)
\xrightarrow{\beta\,\tilde\circ\,\alpha}
(\phi_2,\J_2,\v_2)\,\tilde\circ\,(\psi_2,\K_2,\u_2),
\]
i.e. an arrow $(\psi_1\circ\tilde\pi_1, \K_1\ctimes_\H \J_1, \v_1\circ\tilde\pi_2) \xrightarrow{\beta\,\tilde\circ\,\alpha} (\psi_2\circ\tilde\pi_2, \K_2\ctimes_\H \J_2, \v_2\circ\tilde\pi_2)$.

Unlike the horizontal composition of arrows between morphisms in previous subsection, the construction of horizontal composition of arrows between full-morphisms is slightly subtle. We now describe the construction.

\begin{construction}\label{C hori-compose-full}
Set $\tilde \K_{12}:=\K_1\ctimes_\G \K_2$,$\tilde \L:=\K_1\ctimes_\H\J_1$, $\tilde \J_{12}:=\J_1\ctimes_\H \J_2$, $\tilde \M:=\K_2\ctimes_\H\J_2$, and $\tilde \U:=\sf L\ctimes_G M$. We have the following diagram (comparing with \eqref{E dig-hori-cpose})
\begin{align}\label{E dig-hori-cpose-full}
\begin{split}
\begin{tikzpicture}
\def \x{2.5}
\def \y{0.8}
\node (A00)   at (0,0)        {$\sf G$};
~
\node (A11)   at (\x,\y)      {$\K_1$};
\node (A10)   at (\x,0)       {$\tilde \K_{12}$};
\node (A1-1)  at (\x,-1*\y)   {$\K_2$};
~
\node (A21)   at (2*\x,\y)    {$\sf \tilde L$};
\node (A20)   at (2*\x,0)     {$\sf H$};
\node (A2-1)  at (2*\x,-1*\y) {$\sf \tilde M$};
~
\node (A31)   at (3*\x,\y)    {$\J_1$};
\node (A30)   at (3*\x,0)     {$\tilde \J_{12}$};
\node (A3-1)  at (3*\x,-1*\y) {$\J_2$};
~
\node (A40)   at (4*\x,0)     {$\sf N$};
~
\node at (1.5*\x,0) {$\Downarrow\alpha$};
\node at (3.5*\x,0) {$\Downarrow\beta$};
~
\node (A-10) at (-0.3*\x,0) {$\sf \tilde U$};
~
\path (A00) edge [
<-]  node [auto] {$\scriptstyle{\psi_1}$} (A11);
\path (A00) edge [
<-]  node [auto,swap] {$\scriptstyle{\psi_2}$} (A1-1);
~
\path (A10) edge [->]  node [auto] {$\scriptstyle{\tilde \pi_1}$} (A11);
\path (A10) edge [->]  node [auto,swap] {$\scriptstyle{\tilde \pi_2}$} (A1-1);
~
\path (A11) edge [->]  node [auto] {$\scriptstyle{\u_1}$} (A20);
\path (A1-1) edge [->]  node [auto,swap] {$\scriptstyle{\u_2}$} (A20);
~
\path (A11) edge [
<-]  node [auto] {$\scriptstyle{\tilde \pi_1}$} (A21);
\path (A1-1) edge [
<-]  node [auto,swap] {$\scriptstyle{\tilde \pi_1}$} (A2-1);
~
\path (A21) edge [
->]  node [auto] {$\scriptstyle{\tilde \pi_2}$} (A31);
\path (A2-1) edge [
->]  node [auto,swap] {$\scriptstyle{\tilde \pi_2}$} (A3-1);
~
\path (A20) edge [<-]  node [auto] {$\scriptstyle{\phi_1}$} (A31);
\path (A20) edge [<-]  node [auto,swap] {$\scriptstyle{\phi_2}$} (A3-1);
~
\path (A30) edge [->]  node [auto] {$\scriptstyle{\tilde \pi_1}$} (A31);
\path (A30) edge [->]  node [auto,swap] {$\scriptstyle{\tilde \pi_2}$} (A3-1);
~
\path (A40) edge [
<-]  node [auto,swap] {$\scriptstyle{\v_1}$} (A31);
\path (A40) edge [
<-]  node [auto] {$\scriptstyle{\v_2}$} (A3-1);
~
\path (A-10) edge [
->,bend left=30]node[auto]{$\scriptstyle{\tilde \pi_1}$} (A21);
\path (A-10) edge [
->,bend right=30]node[auto,swap] {$\scriptstyle{\tilde \pi_2}$} (A2-1);
~
\path (A21) edge [
->,bend left=30]
                node[auto]{$\scriptstyle{\v_1\circ\tilde \pi_2}$} (A40);
\path (A2-1) edge [
->,bend right=30]
                node[auto,swap] {$\scriptstyle{\v_2\circ\tilde \pi_2}$} (A40);
\end{tikzpicture}\end{split}
\end{align}
The  arrow $\beta\,\tilde\circ\,\alpha$ we want is a natural transformation $\v_1\circ\tilde \pi_2\circ\tilde \pi_1\stackrel{\beta\,\tilde\circ\,\alpha}{\Longrightarrow} \v_2\circ\tilde \pi_2\circ\tilde \pi_2: \sf \tilde U\rto N$. An object in $\tilde U^0$ is of the form $(k_1,j_1,k_2,j_2)$ with $u^0_1(k_1)=\phi^0_1(j_1)$, $ u^0_2(k_2)=\phi^0_2(j_2)$ in $ H^0$ and $\psi^0_1(k_1)=\psi^0_2(k_2)$ in $ G^0$. It is mapped by $\v_1\circ\tilde \pi_2\circ\tilde \pi_1$ and $\v_2\circ\tilde \pi_2\circ\tilde \pi_2$ respectively to $v^0_1(j_1)$ and $v^0_2(j_2)$. From $\psi^0_1(k_1)=\psi^0_2(k_2)$ we see that $(k_1,k_2)\in \tilde K_{12}^0$, hence we get an arrow in $ H^1$ from the arrow $(\psi_1,\K_1,\u_1)\xrightarrow{\alpha}(\psi_2,\K_2,\u_2)$
\[
\phi^0_1(j_1) = u_1^0(k_1) \xrightarrow{\alpha(k_1,k_2)}
 u_2^0(k_2)=\phi_2^0(j_2).
\]
This gives us an object
\[
\xymatrix{j_1\ar@{.>}[rr]^{\alpha(k_1,k_2)}& & j_2}
\in (\J_1\times_{\phi_1,\H,\phi_2}\J_2)^0.
\]
Only if $\phi_1^0(j_1)=\phi_2^0(j_2)$ and $\alpha(k_1,k_2)=1_{\phi_1^0(j_1)}$ we get an object
\[
\xymatrix{j_1\ar@{.>}[rr]^{\alpha(k_1,k_2)} && j_2}
\in\tilde J_{12}^0=(\J_1\ctimes_{\phi_1,\H,\phi_2}\J_2)^0.
\]
In general this is not the case. However since by Lemma \ref{L Go*/HK-eq-G*/HK}, the natural strict morphism $\q:\tilde \J_{12}\rto \J_1\times_{\phi_1,\H,\phi_2}\J_2$ is an equivalence, we could get an arrow in $\tilde \J_{12}$ as follow.

Since $\phi_1$ and $\phi_2$ are both full-equivalences, there are $j_{1,2}\in  J_2^0$, and $j_{2,1}\in J_1^0$  such that
\[
\phi_1^0(j_{2,1})=\phi_2^0(j_2),\qq \mbox{and}\qq
\phi_2^0(j_{1,2})=\phi_1^0(j_1).
\]
Therefore $(j_1,j_{1,2}),\; (j_{2,1},j_2)\in \tilde J_{12}^0$. Via the equivalence $\phi_1\circ\tilde \pi_1$, (by Lemma \ref{L eqs-compose-eq}, $\phi_1\circ\tilde \pi_1$ is an equivalence), these two objects in $\tilde J_{12}^0$ are mapped respectively to $\phi_1^0(j_1)$ and $\phi^0_1(j_{2,1})=\phi^0_2(j_2)$, which are connected by $\alpha(k_1,k_2)$. Hence by Remark \ref{R equiv-bij-arrows} there is a unique arrow in $\tilde J_{12}^1$
\[
(j_1,j_{1,2}) \xrightarrow{[(\phi_1\circ\tilde \pi_1)^1]\inv(\alpha(k_1,k_2))} (j_{2,1},j_2).
\]
Denote by $(x_\alpha,y_\alpha)=[(\phi_1\circ\tilde \pi_1)^1]\inv(\alpha(k_1,k_2))$.

\begin{rem}
In fact
\[
x_\alpha=(\phi_1^1)\inv(\alpha(k_1,k_2)),\qq y_\alpha=(\phi_2^1)\inv(\alpha(k_1,k_2)).
\]
We now explain this. Since $\phi_1$ and $\phi_2$ are both full-equivalences, then from $(j_1,j_{2,1})$ and $(j_{1,2},j_2)$ we get unique arrows
\[
(\phi_1^1)\inv(\alpha(k_1,k_2)):j_1\rto j_{2,1},\qq\mbox{and} \qq  (\phi_2^1)\inv(\alpha(k_1,k_2)):j_{1,2}\rto j_2.
\]
Then by the bijections of arrows under equivalence we have
\[
((\phi_1^1)\inv(\alpha(k_1,k_2)),(\phi_2^1)\inv(\alpha(k_1,k_2)))
=[(\phi_1\circ\tilde \pi_1)^1]\inv(\alpha(k_1,k_2))
=(x_{\alpha},y_{\alpha}).
\]
We can also use $\phi_2\circ\tilde \pi_2$ to get $[(\phi_2\circ\tilde \pi_2)^1]\inv(\alpha(k_1,k_2))$. However we get nothing else but
\[
[(\phi_1\circ\tilde \pi_1)^1]\inv(\alpha(k_1,k_2))
=[(\phi_2\circ\tilde \pi_2)^1]\inv(\alpha(k_1,k_2))=(x_\alpha,y_\alpha).
\]
\end{rem}

Let us continue the construction. By applying the natural transformation $\beta$ to the arrow $(x_\alpha,y_\alpha)$ we get a commutative diagram in $ N^1$
\[
\xymatrix{
 v_1^0(j_1)\ar[rr]^{\beta(j_1,j_{1,2})}\ar[d]_{ v_1^1(x_\alpha)}&&
 v_2^0(j_{1,2})\ar[d]^{ v_2^1(y_\alpha)}\\
 v_1^0(j_{2,1})\ar[rr]^{\beta(j_{2,1},j_2)}&&
 v_2^0(j_2)}
\]
We define $\beta\,\tilde\circ\,\alpha:\tilde U^0 \rto N^1$ to be
\[
\beta \,\tilde\circ\,\alpha(k_1,j_1;k_2,j_2):=
\beta(j_1,j_{1,2})\cdot v_2^1(y_\alpha)=
v_1^1(x_\alpha)\cdot\beta(j_{2,1},j_2).
\]
It is direct to see that this definition does not depend on the choices of $j_{1,2}$ and $j_{2,1}$, and gives us an arrow $(\phi_1,\L_1,\v_1)\circ(\psi_1,\K_1,\u_1) \xrightarrow{\beta\,\tilde\circ\,\alpha} (\phi_2,\L_2,\v_2)\circ(\psi_2,\K_2,\u_2)$. This finishes the construction.
\end{construction}

\begin{rem}\label{R circ-to-tilde-circ}
We can also get $\beta\,\tilde\circ\,\alpha$ via the horizontal composition in $\mathrm{Mor}^1$ and $i^1$ in Theorem \ref{T equi-i}. The procedure is similar to the way to get $\tilde \bullet$ from $\bullet$ in Remark \ref{R bullet-to-tilde-bullet}.  Given $\alpha$ and $\beta$ as above, we get $i^1(\alpha)$ and $i^1(\beta)$. Then we have
\[
\beta\,\tilde\circ\,\alpha=(i^1(\beta)\circ i^1(\alpha))\circ q^0
\]
with $q^0:\tilde U^0\rto U^0$.
\end{rem}

Via this remark we have similar results for $\,\tilde\circ\,$ as Lemma \ref{L hori-functor}.

\begin{lem}\label{L hori-functor-full}
Combining with composition of full-morphisms we also get a horizontal composition functor
\[
\tilde\circ\,:\sf FMor(G,H)\times FMor(H,M)\rto FMor(G,N)
\]
i.e. the vertical and horizontal composition of arrows between full-morphisms are compatible. Therefore $\,\tilde\circ\,$ is a strict groupoid morphism.
\end{lem}

Similar to Lemma \ref{L horizontal-asso} we have
\begin{lem}\label{L hori-compose-full-asso}
The horizontal composition functor $\,\tilde\circ\,$ is associative under the canonical isomorphism of strict fiber product of three groupoids in Lemma \ref{L fib-prdct-asso}.
\end{lem}

\section{Automorphism groupoids}\label{S gpaction}

In this section we  study the morphism groupoid $\sf Mor(G,G)$ of a groupoid $\sf G$.

\subsection{Center of a groupoid}

To study the automorphisms of groupoids we introduce a new concept of centers of groupoids. We first recall the concept of groupoid action on spaces.

\begin{defn}
For a groupoid $\sf G$ and a space $X$, a {\em (left) $\sf G$-action} on $X$ consists of
\begin{itemize}
\item a map, called the {\em anchor} map, $\rho:M\rto  G^0$,

\item a {\em action map} $\mu: G^1\times_{s, G^0,\rho}M\rto M$
      satisfying
      \[
      \rho(\mu(x,p))=t(x),\qq
      \mu(1_a,p)=p,
      \qq
      \mbox{and}
      \qq
      \mu(x,\mu(y,p))=\mu(y\cdot x,p)
      \]
      whenever the terms are well defined.
\end{itemize}
\end{defn}

Given an action of $\sf G$ on $M$ there is an induced groupoid $\msf G\ltimes M$ with
\[
(\msf G\ltimes M)^0=M,\qq (\msf G\ltimes M)^1= G^1\times_{s, G^0,\rho} M,
\]
and source and target maps given by
\[
s(x,p)=p, \qq t(x,p)=\mu(x,p).
\]
Other structure maps are obvious.

For a groupoid $\sf G$ and an object $a\in  G^0$ the {\em isotropy group of $a$ in $\sf G$} is $\Gamma_a:=G^1(a,a)$, which is a group. Denote by $Z(\Gamma_a)$ the center of the isotropy group $\Gamma_a$. Set
\[
ZG^0=\bigcup_{a\in G^0} Z(\Gamma_a)\subseteq G^1.
\]
There is a $\sf G$-action on $ZG^0$, whose anchor map and action map are
\begin{align*}
\rho:ZG^0\rto G^0,&\qq x\mapsto s(x)=t(x), \\
\mu:G^1\times_{s,G^0,\rho} ZG^0\rto ZG^0,&\qq
(y,x)\mapsto y\inv\cdot x \cdot y.
\end{align*}

\begin{defn}
We define the {\em center groupoid} of $\sf G$ as ${\sf ZG:=G} \ltimes ZG^0$.
\end{defn}

There is a natural strict morphism $\pi:\sf ZG\rto G$ with $\pi^0=\rho$ and $\pi^1$ given by
\begin{eqnarray*}
\pi^1:G^1\times_{s,G^0,\rho} ZG^0\rto G^1,&\qq&(y,x)\mapsto y.
\end{eqnarray*}

\begin{defn}
By a {\em section} of $\pi:\sf ZG\rto G$ we means a section $\sigma:G^0\rto ZG^0$ of the projection $\pi^0: ZG^0\rto G^0$ such that it is invariant under the $\sf G$-action in the meaning of that for every arrow $x:a\rto b$ in $ G^1$
\begin{align}\label{E defeq-mcK(G)}
\sigma(b)=\mu(x,\sigma(a))=x\inv\cdot\sigma(a)\cdot x,\qq i.e. \qq  x\cdot \sigma(b)=\sigma(a)\cdot x.
\end{align}
\end{defn}

We denote by $\mc K(\sf G)$ the set of sections of $\pi:\sf ZG\rto G$. It is easy to see that

\begin{lem}
$\mc K(\sf G)$ is a group.
\end{lem}
\begin{proof}
The multiplication is induced from the composition of arrows in $ G^1$. The identity for the multiplication is the unit section $1:G^0\rto ZG^0$, $a\mapsto (a,1_a)$.
\end{proof}
We call $\mc K(\sf G)$ to be the {\em center} of the groupoid $\sf G$.

\subsection{Automorphisms}

\begin{defn}\label{D automorphisms}
Let $\sf (\psi,K,u)\in \mathrm{Mor}^0(\sf G, G)$. If there exists another morphism $\sf (\phi,L,v)\in \mathrm{Mor}^0(\sf G, G)$ and two arrows
\[
\sf (\psi,K,u)\circ (\phi,L,v)\xrightarrow{\alpha}   1_G,\qq (\phi,L,v)\circ (\psi,K,u) \xrightarrow{\beta} 1_G
\]
in $\mathrm{Mor}^1(\sf G,G)$, we call $\sf (\psi,K,u)$ an {\em automorphism} of $\sf G$. So is $\sf (\phi,L,v)$.
\end{defn}

Let $\mathrm{Aut}^0(\sf G)$ be the set of automorphisms of $\sf G$ and $\mathrm{Aut}^1(\sf G)$ be the induced arrows from $\mathrm{Mor}^1(\sf G,G)$, i.e. we have the following groupoid
\[
{\sf{Aut}(G)} =(\mathrm{Aut}^1(\G)\rightrightarrows \mathrm{Aut}^0(\G)) ={\sf Mor(G,G)}|_{\mathrm{Aut}^0(\sf G)}.
\]

The main theorem of this section is:

\begin{theorem}\label{T |aut|-is-group}
$\sf Aut(G)$ is a $\mc K(\sf G)$-gerbe over its coarse space $|\sf Aut(G)|$. Moreover, $|\sf{Aut}(G)|$ is a group.
\end{theorem}

The proof of this theorem consists of \S \ref{Subs isotropy-aut} (See Corollary \ref{C aut-is-gerbe}) and \S \ref{Subs group-|aut|}.

\subsection{Group action on trivial center topological groupoids}

Motived by  Theorem \ref{T |aut|-is-group} we may consider group actions on topological groupoids.

\begin{defn}\label{D aut-group}
The {\em automorphism group} $\mathrm{Aut}(\sf G)$ of $\sf G$ is defined to be $|\sf Aut(G)|$.
\end{defn}

\begin{exa}[Automorphism groupoid of classifying groupoid]
Consider the classifying groupoid $[\bullet/G]:=(G\rrto \bullet)$ of a group $G$. The automorphism groupoid is equivalent to the action groupoid $G\ltimes \mathrm{Aut}(G)$, where $\mathrm{Aut}(G)$ is the group of automorphisms of $G$ and $G$ acts on it by conjugation. This is a  $Z(G)$-gerbe over the coarse space
\[
\mathrm{Aut}(G)/(G/Z(G)) =\mathrm{Aut}(G)/\mathrm{Inn}(G)=\mathrm{Out}(G),
\]
where $\mathrm{Inn}(G)$ and $\mathrm{Out}(G)$ are the group of inner and outer automorphisms of $G$.
\end{exa}

Now suppose that $\G$ is a groupoid with trivial $\mc K(\G)$. Then the automorphism groupoid $\sf Aut(G)$ is equivalent to the group $|\sf Aut(G)|$. This observation leads to the following definition.

\begin{defn}\label{D group-action}
Let $K$ be a group and $\G$ be a groupoid with trivial $\mc K(\G)$. A {\em $K$-action} on $\sf G$ is a morphism
\[
(\psi,\msf H,\Phi): K\times \sf G \rhpu G
\]
satisfying the following two conditions:
\begin{enumerate}
\item[(1)] For every $k\in K$, the composition $\{k\}\times\G\stackrel{i_k}{\hrto} K\times\G\stackrel{(\psi,\H,\Phi)}{\rhpu} \G$ induces an automorphism of $\G$, where $i_k$ is the natural embedding. This defines a map $\tilde\Phi: K\rto \mathrm{Aut}^0(\G)$.

\item[(2)] $|\tilde\Phi|: K\rto |\mathrm{Aut}(\sf G)|$ is a group homomorphism.
\end{enumerate}
\end{defn}

\subsection{Isotropy groups of automorphisms}\label{Subs isotropy-aut}

\begin{prop}
\label{P isotropy-of-general-auto-u}
For a $\sf (\psi,K,u): G\rhpu G$ in $\mathrm{Aut}^0(\sf G)$, there is a group isomorphism
\begin{align}\label{E def-Psi}
\Psi:\mc K(\G) \xrightarrow{\cong}\Gamma_{\sf (\psi,K,u)}, \qq \sigma\mapsto \sf \sigma\star 1_{(\psi,K,u)},
\end{align}
where $\sigma\star \sf 1_{(\psi,K,u)}$ is defined by \eqref{E def-sigma-1u} in the proof. Hence $\Gamma_{\sf (\psi,K,u)}$ is canonically isomorphic to $\mc K(\sf G)$.
\end{prop}

First we find that automorphisms of $\sf G$ have the following nice properties.
\begin{lem}\label{L aut-u-|u|-surj}
Suppose ${\sf (\psi,K,u),(\phi,L,v)}\in\mathrm{Aut}^0(\G)$, and
\[
\sf (\psi,K,u)\circ (\phi,L,v)\xrightarrow{\alpha} 1_G,\qq
\sf (\phi,L,v)\circ (\psi,K,u)\xrightarrow{\beta} 1_G.
\]
Then the strict morphisms
\begin{align*}
\u\circ \pi_2\circ\pi_1&:{\sf (L\times_GK)\times_GG\rto G}, \qq\mbox{and}\qq \v\circ \pi_2\circ\pi_1:\sf (K\times_G L)\times_GG \rto G
\end{align*}
are both equivalences. Consequently
\[
\msf u^1: K^1(k_1,k_2)\rto G^1(u^0(k_1),u^0(k_2)),
\qq\mbox{and}\qq
|\msf u|: |\sf K|\rto |\sf G|
\]
are both surjective. Same properties hold for $\sf v$.
\end{lem}

Now we proceed to prove the proposition.

\begin{proof}[Proof of Proposition \ref{P isotropy-of-general-auto-u}]
First of all we define the arrow $\Psi(\sigma)=\sigma\star \sf 1_{(\psi,K,u)}\in \Gamma_{\sf (\psi,K,u)}$. Since it is an arrow from $\sf (\psi,K,u)$ to it self, it is defined over $({\sf K\times_G K})^0=\{(k_1,x,k_2)\mid x:\psi^0(k_1)\rto\psi^0(k_2)\mbox{ in } G^1\}$. It is given by
\begin{align}\label{E def-sigma-1u}
{\sf \sigma\star 1_{(\psi,K,u)}}(k_1,x,k_2)
:&=\sigma( u^0(k_1))\cdot 1_{\sf(\psi,K,u)}(k_1,x,k_2)\\
&=\sigma( u^0(k_1))\cdot u^1((\psi^1)\inv(x))
&\mbox{(Eq.\eqref{E def-1u})}.\nonumber\\
&=u^1((\psi^1)\inv(x))\cdot\sigma( u^0(k_2)),
&\mbox{(Eq.\eqref{E defeq-mcK(G)})}.\nonumber
\end{align}
where $(\psi^1)\inv(x): k_1\rto k_2$ is the unique arrow in $K^1(k_1,k_2)$ that is mapped to $x$ by $\psi^1$. We next show $\sigma\star \sf 1_{(\psi,K,u)}$ belongs to $\Gamma_{\sf {(\psi,K,u)}}$, i.e. it is a natural transformation $\u\circ\pi_1\Rto\u\circ\pi_2:\sf K\times_G K\rto G$. Take an arrow in $({\sf K\times_GK})^1$
\begin{align}\label{E arrow-in-K-Sec9.3}
(a,(k_1,x,k_2),b):(k_1,x,k_2)\rto (\tilde k_1,\tilde x,\tilde k_2).
\end{align}
Hence $\psi^1(a)\cdot \tilde x=x\cdot\psi^1(b)$ in $G^1$. Consequently $a\cdot(\psi^1)\inv(\tilde x)=(\psi^1)\inv(x)\cdot b$ in $K^1$ and $u^1(a):u^0(k_1)\rto u^0(\tilde k_1)$ in $G^1$. Then we have
\begin{align*}
\sigma\star {\sf 1_{(\psi,K,u)}}(k_1,x,k_2)\cdot  u^1(b)
&=\sigma( u^0(k_1))\cdot  u^1((\psi^1)\inv(x))\cdot  u^1(b)\\
&=\sigma( u^0(k_1))\cdot  u^1((\psi^1)\inv(x)\cdot b)\\
&=\sigma( u^0(k_1))\cdot  u^1(a\cdot (\psi^1)\inv(\tilde x))\\
&=\sigma( u^0(k_1))\cdot  u^1(a)\cdot u^1((\psi^1)\inv(\tilde x))\\
&= u^1(a)\cdot\sigma( u^0(\tilde k_1))\cdot
 u^1((\psi^1)\inv(\tilde x)) &\mbox{(Eq.\eqref{E defeq-mcK(G)})}\\
&= u^1(a)\cdot \sigma\star {\sf 1_{(\psi,K,u)}}
(\tilde k_1,\tilde x,\tilde k_2).
\end{align*}
Therefore $\sigma\star \sf 1_{(\psi,K,u)}\in \Gamma_{\sf {(\psi,K,u)}}$, and hence $\Psi$ is well defined. We next show that $\Psi$ is a group homomorphism.

For two sections $\sigma,\delta\in\mc K(\sf G)$ we have \begin{align*}
&(\sigma\star \msf 1_{\sf {(\psi,K,u)}})\bullet
(\delta\star \msf 1_{\sf {(\psi,K,u)}})(k_1,x,k_2)\\
&=\sigma\star \msf 1_{\sf {(\psi,K,u)}}(k_1,x_1,k_2')\cdot
  \delta\star \msf 1_{\sf {(\psi,K,u)}}(k_2',x_2,k_2)\\
&=\sigma( u^0(k_1))\cdot  u^1((\psi^1)\inv(x_1))
\cdot  u^1((\psi^1)\inv(x_2))\cdot\delta( u^0(k_2))\\
&=\sigma( u^0(k_1))\cdot
 u^1((\psi^1)\inv(x_1)\cdot (\psi^1)\inv(x_2))\cdot\delta( u^0(k_2))\\
&=\sigma( u^0(k_1))\cdot
 u^1((\psi^1)\inv(x_1\cdot x_2))\cdot\delta( u^0(k_2))\\
&=\sigma( u^0(k_1))\cdot
 u^1((\psi^1)\inv(x))\cdot\delta( u^0(k_2))\\
&=\sigma( u^0(k_1))
\cdot\delta( u^0(k_1)) \cdot u^1((\psi^1)\inv(x))\\
&=(\sigma\cdot\delta)(u^0(k_1)) \cdot u^1((\psi^1)\inv(x))\\
&=(\sigma\cdot\delta)\star \msf 1_{\sf {(\psi,K,u)}}(k_1,x,k_2),
\end{align*}
where $x=x_1\cdot x_2$. Hence  $\Psi(\sigma)\bullet\Psi(\delta)=\Psi(\sigma\cdot\delta)$, and $\Psi$ is a group homomorphism.

We next construct the inverse map $\Phi$ of $\Psi$. Given an arrow $\sf {(\psi,K,u)}\xrightarrow{\alpha} {(\psi,K,u)}$, by applying both $\alpha$ and $\msf 1_{\sf {(\psi,K,u)}}$ to the arrow \eqref{E arrow-in-K-Sec9.3} in $ (\sf K\times_GK)^1$, we get two commutative diagrams
\[\begin{tikzpicture}
\node at (0,0)
{$\xymatrix{
 u^0(k_1)\ar[rr]^{\alpha(k_1,x,k_2)} \ar[d]_{ u^1(a)} &&
 u^0(k_2)\ar[d]^{ u^1(b)} \\
 u^0(\tilde k_1)\ar[rr]^{\alpha(\tilde k_1,\tilde x,\tilde k_2)} &&
 u^0(\tilde k_2),
}$};
\node at (3,0) {$\mbox{and}$};
\node at (6,0)
{$\xymatrix{
 u^0(k_1)
\ar[rr]^{\msf 1_{\sf {(\psi,K,u)}}(k_1,x,k_2)}
\ar[d]_{ u^1(a)}
&&
 u^0(k_2)
\ar[d]^{ u^1(b)}\\
 u^0(\tilde k_1)
\ar[rr]^{\msf 1_{\sf {(\psi,K,u)}}(\tilde k_1,\tilde x,\tilde k_2)}
&&
 u^0(\tilde k_2).}$};
\end{tikzpicture}\]
Consequently we have
\[
\alpha(k_1,x,k_2)\cdot \msf 1_{\sf {(\psi,K,u)}}(k_1,x,k_2)\inv\cdot  u^1(a)=
 u^1(a)\cdot \alpha(\tilde k_1,\tilde x,\tilde k_2)
\cdot \msf 1_{\sf {(\psi,K,u)}}(\tilde k_1,\tilde x,\tilde k_2)\inv.
\]
By Theorem \ref{L aut-u-|u|-surj}, the map $u^1:K^1(k_1,\tilde k_1)\rto G^1( u^0(k_1), u^0(\tilde k_1))$ is surjective. Therefore
\[
\alpha(k_1,x,k_2)\cdot \msf 1_{\sf {(\psi,K,u)}}(k_1,x,k_2)\inv\cdot y=
y\cdot \alpha(\tilde k_1,\tilde x,\tilde k_2)
\cdot \msf 1_{\sf {(\psi,K,u)}}(\tilde k_1,\tilde x,\tilde k_2)\inv
\]
for every $y\in G^1$ with $s(y)= u^0(k_1)$. In particular when $k_1=k_2=k$, $x=1_{\psi^0(k)}$, $\tilde k_1=\tilde k_2=\tilde k$, $\tilde x=1_{\psi^0(\tilde k)}$, for every $y: u^0(k)\rto  u^0(\tilde k)$ we have
\begin{align}\label{E beta-comute-x-in-lemma-Gamma-u}
\alpha(k,1_{\psi^0(k)},k)\cdot y
=y\cdot \alpha(\tilde k,1_{\psi^0(\tilde k)},\tilde k).
\end{align}
Therefore by taking $k=\tilde k$ we see that for every $k\in  K^0$, $\alpha(k,1_{\psi^0(k)},k)\in Z(\Gamma_{ u^0(k)})$.

For every $a\in \mbox{Im}\,( u^0)$ take a pre-image $k\in  K^0$ of $a$ under $u^0$. We first define $\Phi(\alpha)$ on $\mbox{Im}\, u^0$ by
\begin{align}\label{E def-Phi}
\Phi(\alpha)(a):=\alpha(k,1_{\psi^0(k)},k)\in Z(\Gamma_a).
\end{align}
This is independent on the choices of $k$. Suppose there is another $k'\in K^0$ satisfying $u^0(k')=a$. Then since $u^1:K^1(k,k')\rto G^1(a,a)$ is surjective, there is an arrow $x:k\rto k'$ in $K^1$ satisfying $u^1(x)=1_a$, which gives us an arrow
\[
(x,(k,1_{\psi^0(k)},k),x):(k,1_{\psi^0(k)},k)\rto (k',1_{\psi^0(k')},k')
\]
in $({\sf K\times_G K})^1$. By applying $\alpha$ to this arrow we get
\[
\alpha(k,1_{\psi^0(k)},k)\cdot u^1(x)=
u^1(x)\cdot \alpha(k',1_{\psi^0(k')},k'),
\]
i.e. $\alpha(k,1_{\psi^0(k)},k)=\alpha(k',1_{\psi^0(k')},k')$.

We next extend $\Phi(\alpha)$ to the whole $G^0$. Since by Lemma \ref{L aut-u-|u|-surj} $|\sf u|:|K|\rto |G|$ is surjective, every object $b\in  G^0$ is connected to an object $a= u^0(k)\in \mbox{Im}\,u^0$ by an arrow $x:a\rto b$. We then extend $\Phi(\alpha)$ to the whole $G^0$ by
\[
\Phi(\alpha)(b):=x\inv\cdot\Phi(\alpha)(a)\cdot x.
\]
One can see that this is similar to the construction of $ i^1$ in the proof of Theorem \ref{T equi-i}. It is also direct to check that the definition of $\Phi(\alpha)$ does not depend on various choices and it is indeed a section in $\mc K(\sf G)$. It is direct to see that $\Phi$ is the inverse map of $\Psi$. This finishes the proof.
\end{proof}

This proposition implies that all objects of the groupoid $\sf Aut(G)$ have isomorphic isotropy groups. Consider the following subset of $\mathrm{Aut}^1(\sf G)$
\[
\ker {\sf Aut(G)}:=\{\alpha\in \mathrm{Aut}^1(\msf G)\mid s(\alpha)=t(\alpha)\}.
\]
Then source and target maps restrict to a projection $s=t:\ker{\sf Aut(G)}\rto \mathrm{Aut}^0(\sf G)$. From the proof of Proposition \ref{P isotropy-of-general-auto-u} we see that the fiber of $\ker{\sf Aut(G)}\rto \mathrm{Aut}^0(\sf G)$ is a group isomorphic to $\mc K(\sf G)$. On the other hand, from this subset we could construct a new groupoid $(\widetilde{\mathrm{Aut}}^1(\msf G)\rrto \mathrm{Aut}^0(\sf G))$ with
\[
\widetilde{\mathrm{Aut}}^1(\msf G):={\mathrm{Aut}}^1(\msf G)/\ker{\sf Aut(G)},
\]
where the quotient is taken by identifying an arrow $\alpha\in \mathrm{Aut}^1(\msf G)$ with an arrow $\beta\bullet \alpha$ for $\beta\in \ker{\sf Aut(G)}$ satisfying $s(\alpha)=s(\beta)$. Denote the corresponding projection to quotient set by $p^1:{\mathrm{Aut}}^1(\msf G)\rto \widetilde{\mathrm{Aut}}^1(\msf G)$. Then we have a strict morphism
\begin{align}\label{E (p1,id)-to-tilde-aut}
(p^1,id_{\mathrm{Aut}^0(\sf G)}): {\sf Aut(G)}\rto (\widetilde{\mathrm{Aut}}^1(\msf G)\rrto \mathrm{Aut}^0(\sf G)),
\end{align}
which is surjective over arrows and the kernel of $p^1$ is $(p^1)\inv(u(\mathrm{Aut}^0({\sf G})))=\ker {\sf Aut(G)}$. Note that $p^1|_{\ker {\sf Aut(G)}}=s=t$, hence $p^1:\ker p^1\rto \mathrm{Aut}^0(\sf G)$ is a projection with fiber isomorphic to $\mc K(\sf G)$. Therefore the groupoid $\sf Aut(G)$ is a $\mc K(\G)$-gerbe\footnote{Here by a $\mc K(\sf G)$-gerbe we mean a set level gerbe. It consists of a strict morphism $p=(p^0,p^1):\sf G\rto H$ of groupoids such that object sets $G^0=H^0$, maps $p^0=id_{G^0}$ and the fibers of the kernel $\ker p^1\rto G^0$ is isomorphic to $\mc K(\sf G)$ as groups.} over $(\widetilde{\mathrm{Aut}}^1(\msf G)\rrto \mathrm{Aut}^0(\sf G))$. By Proposition \ref{P isotropy-of-general-auto-u} we see that $(\widetilde{\mathrm{Aut}}^1(\msf G)\rrto \mathrm{Aut}^0(\sf G))$ is equivalent to the trivial groupoid $(|\sf Aut(G)|\rrto |\sf Aut(G)|)$ which represent the space $|\sf Aut(G)|$.  Therefore
\begin{cor}\label{C aut-is-gerbe}
The groupoid $\sf Aut(G)$ is a $\mc K(\G)$-gerbe over its coarse space $|\sf Aut(G)|$.
\end{cor}

\subsection{Group structure over $|\sf Aut(\G)|$}
\label{Subs group-|aut|}

In this section we show that the coarse space $|\sf Aut(G)|$ of the automorphism groupoid of $\G$ is a group. The proof consists of the following five lemmas.

\begin{lem}[Multiplication] \label{L circ-on-|Mor|}
The composition $\circ$ induces a multiplication over the coarse space $|\sf Mor(G,G)|$.
\end{lem}
\begin{proof}
Let $\sf (\psi,K,u),(\psi',K',u'), (\phi,L,v)\in\mathrm{Mor}^0(G,G)$ and $\sf (\psi,K,u)\xrightarrow{\alpha} (\psi',K',u')$ be an arrow. Then we have two arrows (see Construction \ref{C hori-compose})
\begin{align*}
\sf (\psi,K,u)\circ(\phi,L,v)
            \xrightarrow{\alpha\circ 1_{(\phi,L,v)}} (\psi',K',u')\circ(\phi,L,v),\qq
\sf (\phi,L,v)\circ(\psi,K,u)
      \xrightarrow{1_{(\phi,L,v)}\circ\alpha} (\phi,L,v)\circ(\psi',K',u').
\end{align*}
Therefore we get a well-defined multiplication  $|\sf (\psi,K,u)|\circ|(\phi,L,v)|=|\sf (\psi,K,u)\circ(\phi,L,v)|$ on $|\sf Mor(G,G)|$. This finishes the proof.
\end{proof}

\begin{lem}[Associativity]\label{L Asso-|Aut|}
The induced multiplication over $|\sf Mor(\G,\H)|$ is associative.
\end{lem}

\begin{proof}
For simplicity, here we denote morphisms by a single word. The associativity means that every triple $\sf A,B,C$ of morphisms in $\mathrm{Mor}^0(\sf G)$ satisfies
\[
\sf (|A|\circ |B|)\circ |C|=|A|\circ(|B|\circ |C|).
\]
We have showed in Lemma \ref{L horizontal-asso} that via the isomorphism of fiber products of a triple of groupoids, the composition $\sf (A\circ B)\circ C$ is identified with $\sf A\circ(B\circ C)$. Then it is direct to construct an arrow $\sf (A\circ B)\circ C\rto \sf A\circ(B\circ C)$. The lemma follows.
\end{proof}

\begin{lem}[Identity]\label{L identity-|Aut|}
The identity in with respect to the multiplication over $|\sf Mor(G,G)|$ is the image of $\sf 1_G=(id_G,G,id_G)$ in $|\sf Mor(G,G)|$.
\end{lem}

\begin{proof}
We construct two arrows $\sf 1_G\circ (\psi,K,u)\xrightarrow{\alpha_{1,(\psi,K,u)}} (\psi,K,u)$, and $\sf (\psi,K,u)\circ 1_G\xrightarrow{\alpha_{(\psi,K,u),1}} (\psi,K,u)$ for every automorphism $\sf (\psi,K,u)\in\mathrm{Aut}^0(G)$. The composed morphism $\sf (\psi,K,u)\circ 1_G$ is
\[
\G\xleftarrow{{\sf id_G}\circ \pi_1} \W:=\G\times_{{\sf id_G,G,\psi}}\K
\xrightarrow{\u\circ\pi_2} \H.
\]
The arrow $\sf \alpha_{(\psi,K,u),1}$ we want is a natural transformation in the diagram
\[
\begin{tikzpicture}
\def \x{3}
\def \y{0.8}
\node (A00)  at    (0,0)        {$\sf G$};
\node (A11)  at    (\x,\y)      {$\sf W$};
\node (A10)  at    (\x,0)       {$\sf L$};
\node (A1-1) at    (\x,-1*\y)   {$\sf K$};
\node (A20)  at    (2.5*\x,0)     {$\sf H$};
\node at (1.7*\x,0) {$\sf \Downarrow\alpha_{(\psi,K,u),1}$};
\path (A00) edge [<-] node [auto] {$\scriptstyle{{\sf id_G}\circ\pi_1}$} (A11);
\path(A00)edge[<-]node[auto,swap]{$\scriptstyle{\psi}$}(A1-1);
\path (A10) edge [->] node [auto] {$\scriptstyle{\pi_1}$} (A11);
\path (A10) edge [->] node [auto,swap] {$\scriptstyle{\pi_2}$} (A1-1);
\path (A11) edge [->] node [auto] {$\scriptstyle{\u\circ\pi_2}$} (A20);
\path (A1-1) edge [->] node [auto,swap] {$\scriptstyle{\sf u}$} (A20);
\end{tikzpicture}
\]
with $\L=(L^1\rrto L^0):=\W\times_{{\sf id_G}\circ \pi_1,\G,\psi} \K$. Elements in $L^0$ is of the form
\[
\xymatrix{k'&(g \ar@{.>}[l]_y^\G  \ar@{.>}[r]^x_\G &k),}
\]
with $x:g\rto \psi^0(k)$, and $y: g\rto \psi^0(k')$. We denote it by $(k',y,g,x,k)$. From this object we get $y\inv\cdot x: \psi^0(k')\rto \psi^0(k)$. Since $\psi$ is an equivalence, we get a unique arrow $(\psi^1)\inv(y\inv\cdot x):k'\rto k$. We set
\[
\alpha_{\sf(\psi,K,u),1}(k',y,g,x,k)= u^1((\psi^1)\inv(y\inv\cdot x)).
\]
Then it is direct to check that this is the arrow we want. $\alpha_{1,\sf (\psi,K,u)}$ is defined similarly.
\end{proof}

\begin{lem}[Closedness]\label{L closed-|Aut|}
$|\sf Aut(G)|$ is closed with respect to the multiplication on $|\sf Mor(G,G)|$.
\end{lem}

\begin{proof}
For simplicity, here we also denote morphisms by a single word. Suppose we have automorphisms $\sf A, B, A', B'\in \mathrm{Aut}^0(G,G)$ and arrows $\sf A\circ A'\xrightarrow{\alpha} 1_G$, $\sf A'\circ A\xrightarrow{\beta} 1_G$, $\sf B\circ B'\xrightarrow{\gamma} 1_G$, $\sf B'\circ B\xrightarrow{\delta} 1_G$. We next show $A\circ B\in\mathrm{Aut}^0(\G)$. We have arrows
\begin{align*}
&\sf (A\circ B)\circ (B'\circ A') \rto
\sf A\circ (B\circ(B'\circ A')) \rto
\sf A\circ ((B\circ B')\circ A')\\
&\sf \xrightarrow{ 1_A\circ(\gamma\circ 1_{A'})}
A\circ (1_G\circ A') \rto ( A\circ 1_G)\circ A'
\xrightarrow{\alpha_{A,1} \circ 1_{A'}}
A\circ A'\xrightarrow{\alpha} 1_G
\end{align*}
and
\begin{align*}
&\sf (B'\circ A')\circ (A\circ B) \rto
B'\circ (A'\circ(A\circ B)) \rto
B'\circ ((A'\circ A)\circ B)\\
&\sf \xrightarrow{1_{A'}\circ(\beta\circ 1_{B})}
B'\circ (1_G\circ B) \rto (B'\circ 1_G)\circ B
\xrightarrow{\alpha_{B',1}\circ 1_{B}}
B'\circ B\xrightarrow{\delta} 1_G,
\end{align*}
where unmarked arrows are obtained from the Lemma \ref{L horizontal-asso} and $\sf \alpha_{A,1}, \alpha_{B',1}$ are the arrows defined in the proof of last Lemma. Therefore $\sf A\circ B\in \mathrm{Aut}^0(G,G)$, and $|\sf Aut(G)|$ is closed under the multiplication.
\end{proof}

\begin{lem}[Inverse]\label{L inverse-|Aut|}
Every $|u|\in |\sf Aut(G)|$ has an inverse.
\end{lem}

\begin{proof}
This follows from the definition of $\sf\mathrm{Aut}^0(G)$.
\end{proof}

Combing these five lemmas and noting that $\sf 1_G\in \mathrm{Aut^0}(G)$, we finish the proof of that $|\sf Aut(G)|$ is a group.

\appendix


\section{Topology on morphism groupoid}

In this appendix we restrict ourselves to topological groupoids and explain how to assign topology for a morphism groupoid. Since the general case is rather technically complicated and away from central topic of the paper, instead of working for $\sf Mor(G,H)$ or $\sf FMor(G,H)$, we focus ourselves on a natural sub-groupoid of $\sf FMor(G,H)$, which we denote by $\sf OFMor(G,H)$, together with certain assumption on $\sf G$. We point out that such assumption doesn't rule out the case we are interested in. In particular, this appendix explains how one can assign topology for the refinement morphism groupoid of orbifold groupoids and shows that the coarse space of the automorphism groupoid of an orbifold groupoid has a natural topological group structure.

We call a full-equivalence $\psi:\sf K\rto G$ an open refinement, if $K^0$ is a disjoint union of open subsets of $G^0$ and $\sf K$ is the pull back groupoid along the inclusion $K^0\hrto G^0$ with $\psi$ being the corresponding strict morphism of groupoids. We denote by
\[
\mathrm{OFMor}^0({\sf G,H})=\bigsqcup_{\psi:\sf K\rto G,\,\text{open refinement}}\mathrm{SMor^0}(\sf K,H)
\]
the set of open refinements of $\sf G$. By restricting ${\sf FMor(G,H)}$ over $\mathrm{OFMor}^0({\sf G,H})$, we obtain a subgroupoid
\[
{\sf OFMor(G,H)}={\sf FMor(G,H)}\big|_{\mathrm{OFMor}^0({\sf G,H})}
\]
of $\sf FMor(G,H)$. We remark that when $\sf G$ and $\sf H$ are both orbifold groupoids, $\sf OFMor(G,H)$ is Morita equivalent to the groupoid of orbifold homomorphisms (cf. \cite{Adem-Leida-Ruan2007}).

In the following we construct a natural topology over $\sf OFMor(G,H)$ by using compact-open topology. In later context, the space of continuous maps $C(X, Y)$ between two topological space $X$ and $Y$ is always assigned with the compact-open topology.

\subsection{Topology on $\sf SMor(K,H)$}

We first consider the topology over the groupoid
\[
{\sf SMor(K,H)}=(\mathrm{SMor}^1({\sf K,H})\rrto \mathrm{Smor}^0(\sf K,H))
\]
of strict morphisms from $\sf K$ to $\sf H$, where $\mathrm{SMor}^0(\sf K,H)$ is the space of all (continuous) strict morphisms from $\sf K$ to $\sf H$, and $\mathrm{SMor}^1(\sf K,H)$ is the space of all (continuous) natural transformations.

As
\[
\mathrm{SMor}^0({\sf K,H})\subseteq C(K^0,H^0)\times C(K^1,H^1).
\]
We use the induced topology over $\mathrm{SMor}^0({\sf K,H})$.

As
\[
\mathrm{SMor}^1({\sf K,H})\subseteq \mathrm{SMor}^0({\sf K,H})\times C(K^0,H^1),
\]
we use the induced topology over $\mathrm{SMor}^1({\sf K,H})$.

\begin{lem}
$\sf SMor(K,H)$ is a topological groupoid.
\end{lem}
\begin{proof}
We first show that the source and target maps
\[
S,T: \mathrm{SMor}^1({\sf K,H})\rto \mathrm{SMor}^0({\sf K,H})
\]
are continuous. As $S$ is the composition
\[
\mathrm{SMor}^1({\sf K,H})\hrto \mathrm{SMor}^0({\sf K,H})\times C(K^0,H^1) \xrightarrow{\text{proj}_1} \mathrm{SMor}^0({\sf K,H}),
\]
it is continuous. On the other hand, the map $T(\msf u=(u^0,u^1),\sigma)=\msf v=(v^0,v^1)$ is given by
\[
v^0(x)=t(\sigma(x)),\qq
v^1(g)=\sigma(s(g))\inv\cdot u^1(g) \cdot \sigma(t(g)).
\]
Note that $\sigma$ is a section of $(u^0)^*s: (u^0)^* H^1\rto K^0$, and $t\circ\sigma$ is the composition
\[
X^0\xrightarrow{(u^0(x),\sigma(x))} H^1 \xrightarrow{t} H^0.
\]
Hence $t\circ \sigma$ is continuous on $(u^0,\sigma)$ since $t$ is continuous. Similarly, since multiplication, inverse map on $H^1$, $s$ and $t$ are all continuous, $\sigma(s(g))\inv\cdot u^1(g) \cdot \sigma(t(g))$ is continuous on $u^1$.

We next consider the inverse map $I:\mathrm{SMor}^1({\sf K,H})\rto\mathrm{SMor}^1(\sf K,H)$ and the multiplication
\[
M:\mathrm{SMor}^1({\sf K,H})\times_{T,S}\mathrm{SMor}^1(\sf K,H)\rto\mathrm{SMor}^1(\sf K,H).
\]
For the inverse map we have
\[
I(\msf u,\sigma)=(T(\msf u,\sigma),\sigma\inv).
\]
Hence it is continuous, since inverse map on $H^1$ is continuous. For the multiplication map we have
\[
M((\msf u,\sigma),(T(\msf u,\sigma),\sigma'))=
(\msf u,\sigma\cdot\sigma').
\]
It is continuous since the multiplication on $H^1$ is continuous.

Finally, the unit map $U:\mathrm{SMor}^0({\sf K,H})\rto \mathrm{SMor}^1(\sf K,H)$ is given by
\[
U(\sf u)=(\sf u, 1_{\sf u}),
\]
where $1_{\sf u}(x)=1_{u^0(x)}$. Since $1:H^0\rto H^1$ is continuous, $U$ is continuous.
\end{proof}

\subsection{Topology on $\sf OFMor(G,H)$}

Now we consider the topology over $\sf OFMor(G,H)=(\mathrm{OFMor}^1({\sf G,H})\rrto\mathrm{OFMor}^0(\sf G,H))$. As
\[
\mathrm{OFMor}^0({\sf G,H})=\bigsqcup_{\psi:{\sf K\rto G},\, \text{open refinement}} \mathrm{SMor}^0(\sf K,H),
\]
we assign $\mathrm{OFMor}^0({\sf G,H})$ the topology of disjoint union of topologies of each component.

We next define the topology on $\mathrm{OFMor}^1(\sf G,H)$. Given any two open refinements $\psi:\sf K\rto G$ and $\phi:\sf W\rto G$, denote the set of arrows from the component $\mathrm{SMor}^0(\sf K,H)$ to $\mathrm{SMor}^0(\sf W,H)$ by $\mathrm{Mor}^1(\psi,\phi)$. Then
\[
\mathrm{OFMor}^1({\sf G,H})=\bigsqcup_{\psi,\phi}\mathrm{OFMor}^1(\psi,\phi).
\]
We define the topology on $\mathrm{OFMor}^1(\sf G,H)$ to be the disjoint union of topologies of each $\mathrm{OFMor}^1(\psi,\phi)$ described below.

Recall that we have the two projections $\pi_1: {\sf K\ctimes_G W\rto K}$, $\pi_2:{\sf K\ctimes_G W\rto W}$. So we have two continuous map
\[
\pi_1^*:\mathrm{SMor}^0({\sf K,H})\rto \mathrm{SMor}^0(\sf K\ctimes_G W,H),\qq
\pi_2^*:\mathrm{SMor}^0({\sf W,H})\rto \mathrm{SMor}^0(\sf K\ctimes_G W,H).
\]
Consider the map
\[
\pi_1^*\times\pi^*_2:\mathrm{SMor}^0({\sf K,H})\times \mathrm{SMor}^0({\sf W,H}) \rto \mathrm{SMor}^0({\sf K\ctimes_G W,H})\times \mathrm{SMor}^0({\sf K\ctimes_G W,H})
\]
and
\[
S\times T:\mathrm{SMor}^1({\sf K\ctimes_G W,H})\rto
\mathrm{SMor}^0({\sf K\ctimes_G W,H})\times \mathrm{SMor}^0({\sf K\ctimes_G W,H}).
\]
Then $\mathrm{OFMor}^1(\psi,\phi)$ is the fiber product of these two maps, that is
\[
\mathrm{OFMor}^1(\psi,\phi)=\mathrm{SMor}^0({\sf K,H})\times_{\pi^*_1,S}\mathrm{SMor}^1({\sf K\ctimes_G W,H})\times_{T,\pi_2^*} \mathrm{SMor}^0({\sf W,H}).
\]
So it inherits a topology from this fiber product.

\begin{theorem}\label{T top-on-OFMor}
Suppose that $G^0$ is a regular topological space, then $\sf OFMor(G,H)$ is a topological groupoid.
\end{theorem}
\begin{proof}
The source and target maps are
\[
\mathrm{OFMor}^1(\psi,\phi)\xrightarrow{(\text{proj}_1,\,\text{proj}_3)} \mathrm{SMor}^0({\sf K,H})\times \mathrm{SMor}^0({\sf W,H}),
\]
projections to factors, hence are continuous.

We next consider the multiplication map
\[
M:\mathrm{OFMor}^1(\psi,\phi)\times_{T,S}\mathrm{OFMor}^1(\phi,\varphi) \rto \mathrm{OFMor}^1(\psi,\varphi).
\]
Take two arrows $\alpha:\sf u\rto \sf w$ and $\beta:\sf w\rto v$ in $\mathrm{OFMor}^1(\psi,\phi), \mathrm{OFMor}^1(\phi,\varphi)$ respectively. Let $\alpha\,\tilde\bullet\,\beta:\sf u\rto v$ be their multiplication. Take an open neighborhood of $\alpha\,\tilde\bullet\,\beta:\sf u\rto v$:
\[
U_{\sf u}\times_{\pi_1^*,S}U_{\alpha\,\tilde\bullet\,\beta}\times_{T,\pi^*_2} U_{\sf v}.
\]
We next construct two open neighborhoods of $\alpha:\sf u\rto w$ and $\beta:\sf w\rto v$ whose images under $M$ are contained in $U_{\sf u}\times_{\pi_1^*,S}U_{\alpha\,\tilde\bullet\,\beta}\times_{T,\pi^*_2} U_{\sf v}$.

Note that by the definition of topology over $\mathrm{SMor}^1({\sf K\ctimes_G V,H})$,
\[
U_{\alpha\,\tilde\bullet\,\beta}\subseteq \mathrm{SMor}^1({\sf K\ctimes_G V,H})\subseteq\mathrm{SMor}^0({\sf K\ctimes_G V,H})\times C(({\sf K\ctimes_G V})^0,H^1).
\]
So for simplicity we could assume that
\[
U_{\alpha\,\tilde\bullet\,\beta}=
\left(\pi^*_1 U_{\sf u}\times [A,U]\right)\cap \mathrm{SMor}^1({\sf K\ctimes_G V,H}),
\]
where
\[
[A,U]=\{f:({\sf K\ctimes_G V})^0\rto H^1\mid f(A)\subseteq U\}\subseteq C(({\sf K\ctimes_G V})^0,H^1)
\]
is an open set in $C(({\sf K\ctimes_G V})^0,H^1)$ with   $A\subseteq ({\sf K\ctimes_G V})^0$ being a compact subset and $U\subseteq H^1$ being an open subset. We next construct open neighborhoods of $\alpha$ and $\beta$.

For simplicity, we could assume that $A\subseteq K_a\cap V_b$ with $K_a$ a component of $K^0$ and $V_b$ a component of $V^0$. So we can view $A\subseteq G^0$. Now $A$ is covered by $W^0$ via $W^0\hrto G^0$. Without loss of generality we could assume that $A\subseteq W_1\cap W_2$, since $A$ is compact. By the assumption that $G^0$ is regular, for every $x\in A\cap W_i$, there is an open neighborhood $U_{x,i}\subseteq W_i$ of $x$ such that $x\in U_{x,i}\subseteq \overline{U}_{x,i}\subseteq W_i$. These open neighborhoods forms an open cover of $A$. So we get a finite sub-cover of $A$, say $U_{x_1,1},\ldots, U_{x_k,1},U_{x_{k+1},2},\ldots, U_{x_n,2}$. Then we take
\[
A_1:=\left(\bigcup_{j=1}^k \overline U_{x_j,1} \right) \cap A,
\qq
\text{and}
\qq
A_2:=\left(\bigcup_{j=k+1}^n \overline U_{x_j,2} \right) \cap A,
\]
Both $A_1$ and $A_2$ are compact subsets of $G^0$ and covered by $W_1$ and $W_2$ respectively. So $A_1$ and $A_2$ are compact subsets in both $({\sf K\ctimes_G W})^0$ and $({\sf W\ctimes_G V})^0$.

Now consider the open subset $m\inv(U) \subseteq H^1\times_{t,s}H^1$, where $m:H^1\times_{t,s}H^1\rto H^1$ is the multiplication map of $\sf H$. Suppose the projections of $m\inv(U)$ to both factors of $H^1\times_{t,s} H^1$ are $U_1$ and $U_2$. Note that $U_1$ and $U_2$ are both open subsets of $H^1$ since the projections to both factors are open. Then we get two open subsets of $C(({\sf K\ctimes_GW})^0,H^1)$ and $C(({\sf W\ctimes_GV})^0,H^1)$ respectively, which are
\begin{align*}
&[A_1,U_1]\cap [A_2,U_1]=\{f:({\sf K\ctimes_G W})^0\rto H^1\mid f(A_1),f(A_2)\subseteq U_1\},\\
\text{and}\qq
&[A_1,U_2]\cap [A_2,U_2]=\{f:({\sf W\ctimes_G V})^0\rto H^1\mid f(A_1), f(A_2)\subseteq U_2\}.
\end{align*}

Then one see that
\[
\left(\pi^*_1U_{\sf u}\times ([A_1,U_1]\cap [A_2,U_1])\right)\cap \mathrm{SMor}^1({\sf K\ctimes_G W,H})
\]
is an open neighborhood of $\alpha$ in $\mathrm{SMor}^1({\sf K\ctimes_G W,H})$, denoted by $U_\alpha$. Let $U_{\sf w}$ be an open neighborhood of $\sf w$. So
$
U_{\sf u}\times_{\pi^*_1,S}U_{\alpha}\times_{T,\pi^*_2} U_{\sf w}
$
is an open neighborhood of $\alpha:\sf u\rto w$ in $\mathrm{Mor}(\psi,\phi)$. Similarly we get an neighborhood $U_{\beta}$ of $\beta$ in $\mathrm{SMor}^1({\sf K\ctimes_G W,H})$
\[
U_\beta=\left(\pi^*_1U_{\sf w}\times([A_1,U_2]\cap [A_2,U_2])\right)\cap \mathrm{SMor}^1({\sf W\ctimes_G V,H})
\]
and an open neighborhood $U_{\sf w}\times_{\pi^*_1,S} U_\beta\times_{T,\pi^*_2} U_{\sf v}$ of $\beta:\sf w\rto v$ in $\mathrm{Mor}(\phi,\varphi)$. Then we have
\[
M\left(\left(U_{\sf u}\times_{\pi^*_1,S}U_{\alpha}\times_{T,\pi^*_2} U_{\sf w}\right)\times_{\text{proj}_3,\text{proj}_1} \left(U_{\sf w}\times_{\pi^*_1,S} U_\beta\times_{T,\pi^*_2} U_{\sf v}\right)\right)\subseteq U_{\sf u}\times_{\pi^*_1,S} U_{\alpha\,\tilde\bullet\,\beta}\times_{T,\pi^*_2} U_{\sf v}.
\]
Therefore $M$ is continuous.

The inverse map and unit map are obviously continuous.
Therefore $\sf Mor(G,H)$ is a topological groupoid.
\end{proof}

One can see that the assumption that $G^0$ is regular is used to construct a finite cover of a compact set $A$ in terms of compact sets subject to a given open cover of $A$. It is more subtle to construct smooth structure over $\sf OFMor(G,H)$ when $\sf G$ and $\sf H$ are Lie groupoids, and we deal with this issue in \cite{Chen-Du-Ono}.

\subsection{Automorphisms}
We can also define automorphisms in $\mathrm{OFMor}^0(\sf G,G)$ in the same manner as Definition \ref{D automorphisms}. We denote the set of all automorphisms in $\mathrm{OFMor}^0(\sf G,G)$ by $\mathrm{OAut}^0(\msf G)$. So $\mathrm{OAut}^0(\msf G)=\mathrm{OFMor}^0({\sf G,G})\cap\mathrm{Aut}^0(\sf G)$. Then we also get a subgroupoid of $\sf OFMor(G,G)$
\[
{\sf OAut(G):=OFMor(G,G)}\Big|_{\mathrm{OAut}^0(\msf G)}.
\]
Suppose $G^0$ is regular, then $\sf OFMor(G,G)$ is a topological groupoid by Theorem \ref{T top-on-OFMor}. We then assign $\sf OAut(G)$ the induced topology. With the quotient topology, its coarse space $|\sf OAut(G)|$ is also a topological space.

As in \S \ref{S gpaction} we have a group structure over $|\sf OAut(G)|$, and $\sf OAut(G)$ is a (set level) $\mc K(\sf G)$-gerbe over $|\sf OAut(G)|$.

\begin{rem}
Since now $\G$ is a topological groupoid, we could endow $ZG^0$ the subspace topology from the inclusion $ZG^0\subseteq G^1$. Then one can see that the $\G$-action on $ZG^0$ is topological\footnote{A topological action of a topological groupoid on a topological space is a groupoid action for which the anchor and action maps are both continuous.}, hence $\msf {ZG}=\msf G\ltimes ZG^0$ is a topological groupoid. So now $\mc K(\sf G)$ consists of all continuous sections of $\msf {ZG}\rto \G$. We could view $\mc K(\sf G)$ as a subspace of the space of continuous sections of $\pi:ZG^0\rto G^0$, the latter one has the compact-open topology. Then we assign $\mc K(\sf G)$ the induced topology, with which $\mc K(\sf G)$ becomes a topological group.
\end{rem}

\begin{theorem}
Suppose $G^0$ is regular, then $\sf OAut(G)$ is a topological $\mc K(\sf G)$-gerbe over $|\sf OAut(G)|$.
\end{theorem}
Here by a topological $\mc K(\sf G)$-gerbe $\msf p=(p^0,p^1):\sf G\rto H$ we mean that $\msf p$ is a strict morphism of topological groupoids, $p^0=id_{G^0}$, and the kernel $\ker p^1$ is a locally trivial bundle of groups whose fiber is isomorphic to $\mc K(\sf G)$.
\begin{proof}
First as in \eqref{E (p1,id)-to-tilde-aut}, \S \ref{Subs isotropy-aut}, we have the projection
\[
\msf p=(p^1,id_{\mathrm{OAut}^0(\msf G)}):{\sf OAut(G)}\rto (\widetilde{\mathrm{OAut}}^1(\msf G)\rrto \mathrm{OAut}^0(\msf G)).
\]
We only have to show that $p^1:\ker p^1\rto\mathrm{OAut}^0(\sf G)$ is a locally trivial bundle of group with fiber isomorphic to $\mc K(\sf G)$.

By the map $\Psi$ in \eqref{E def-Psi}, Proposition \ref{P isotropy-of-general-auto-u}, we have a map
\[
\tilde\Psi:\mathrm{OAut}^0(\msf G)\times\mc K(\msf G)\rto \ker p^1=\ker \mathrm{OAut}^1(\sf G).
\]
By the same proof used in the proof of Theorem \ref{T top-on-OFMor} we can see that $\tilde\Psi$ is continuous. By the map $\Phi$ in \eqref{E def-Phi}, Proposition \ref{P isotropy-of-general-auto-u}, we see that there is also another map
\[
\tilde\Phi: \ker p^1=\ker \mathrm{OAut}^1(\msf G)\rto \mathrm{OAut}^0(\msf G)\times\mc K(\msf G).
\]
The map $\ker p^1\rto \mathrm{OAut}^0(\G)$ is continuous. We only have to show that the map $\text{proj}_2\circ\tilde\Phi:\ker p^1\rto \mc K(\msf G)$ is continuous. This also follows from the proof used in the proof of Theorem \ref{T top-on-OFMor}. Therefore $\tilde\Phi$ is also continuous. Then by the proof of Proposition \ref{P isotropy-of-general-auto-u} we see that $\tilde\Phi$ is the inverse map of $\tilde\Psi$. So they are both homeomorphisms. This shows that $\ker p^1$ is a globally trivial bundle of groups. Finally, by Proposition \ref{P isotropy-of-general-auto-u}, $\tilde\Psi$ gives the isomorphisms between fibers of $\ker p^1$ and $\mc K(\sf G)$. Therefore $\msf p=(p^1,id_{\mathrm{OAut}^0(\msf G)}):{\sf OAut(G)}\rto (\widetilde{\mathrm{OAut}}^1(\msf G)\rrto \mathrm{OAut}^0(\msf G))$ gives rise to a topological $\mc K(\sf G)$-gerbe over $(\widetilde{\mathrm{OAut}}^1(\msf G)\rrto \mathrm{OAut}^0(\msf G))$. On the other hand, $(\widetilde{\mathrm{OAut}}^1(\msf G)\rrto \mathrm{OAut}^0(\msf G))$ is equivalent to its coarse space $|(\widetilde{\mathrm{OAut}}^1(\msf G)\rrto \mathrm{OAut}^0(\msf G))|=|\sf OAut(G)|$, since for this groupoid the isotropy groups of objects in $\mathrm{OAut}^0(\msf G)$ are all trivial. Therefore, $\sf OAut(G)$ is a $\mc K(\sf G)$-gerbe over $|\sf OAut(G)|$.
\end{proof}

\begin{rem}
Under the assumption that $G^0$ is locally compact, Hausdorff and $\sf G$ is proper \'etale, we explain that $|\sf OAut(G)|$ is a topological group. First notice that, under the current situation there is a continuous map
\[
\mathrm{OFMor}^0({\sf G,G})=\bigsqcup_{\psi:{\sf K\rto G},\,\text{open refinement}}\mathrm{SMor}^0({\sf K,H})\rto \bigsqcup_{\psi:{\sf K\rto G},\,\text{open refinement}}C^0(|\sf K|,|G|)\rto C^0(|\sf G|,|G|).
\]
It is direct to see that this continuous map descends to a continuous map $|{\sf OFMor(G,G)}|\rto C^0(|\sf G|,|G|)$. In particular, we get a continuous map $|{\sf OAut(G)}|\rto \mathrm{Homeo}(|\sf G|)$. From the definition of composition of morphisms in Definition \ref{D comp-full-mor} one can see that $|{\sf OAut(G)}|\rto \mathrm{Homeo}(|\sf G|,|G|)$ is also a group homomorphism. Denote the image by $\mathrm{Homeo}_*(|\sf G|)$.

Under the above assumption $\mathrm{Homeo(|\G|)}$ is a topological group w.r.t compact-open topology. So $\mathrm{Homeo}_*(|\G|)$ is also a topology group with induced topology.

Since $\sf G$ is proper \'etale one can show that $|{\sf OAut(G)}|\rto \mathrm{Homeo}_*(|\sf G|,|G|)$ is an open covering map. So from the commutative diagram
\[
\xymatrix{
|{\sf OAut(G)}|\times |{\sf OAut(G)}| \ar[r] \ar[d] &
|{\sf OAut(G)}| \ar[d]\\
\mathrm{Homeo}_*(|\sf G|,|G|)\times \mathrm{Homeo}_*(|\sf G|,|G|) \ar[r] &
\mathrm{Homeo}_*(|\sf G|,|G|)}
\]
we see that the multiplication on $\sf |OAut(G)|$ is continuous. Similarly, the inverse map over $|\sf OAut(G)|$ is also continuous. Hence $|\sf OAut(G)|$ is a topological group.

Therefore when $\sf G$ is an orbifold groupoid, $|\sf OAut(G)|$ is a topological group.
\end{rem}


\begin{thebibliography}{SKK2}



\bibitem{Adem-Leida-Ruan2007}
                        A. Adem, J. Leida, Y. Ruan,
                        \textit{Orbifolds and stringy topology.}
                        Cambridge University Press, Cambirdge, 2007.

\bibitem{Borzellino-Brunsden2008}
             J. E. Borzellino, V. Brunsden,
             \textit{A manifold structure for the group of orbifold diffeomorphisms of a smooth orbifold}, J. Lie Theory, 18 (2008), 979--1007.

\bibitem{Chen-Du-Ono} B. Chen, C.-Y. Du, K. Ono,  \textit{Groupoid structure of moduli space of orbifold stable maps}, in preparation.

\bibitem{Chen-Ruan2002}
                      W. Chen, Y. Ruan,
                      \textit{Orbifold Gromov--Witten theory},
                      Cont. Math., 310 (2002), 25--86.

\bibitem{Chen-Ruan2004}
                        W. Chen, Y. Ruan,
                        \textit{A new cohomology theory for orbifold},
                        Commun. Math. Phys., 248 (2004), 1--31.

\bibitem{Du-Shen-Zhao2018}
                        C.-Y. Du, L. Shen, X. Zhao,
                        \textit{Spark complexes on good effective orbifold atlases categorically}, Theory Appl. Categ., 33 (2018), 784--812.


\bibitem{Leinster1998}
                        T. Leinster,
                        \textit{Basic bicategories},
                        arXiv:math/9810017, 1998.

\bibitem{Lupercio-Uribe2004} E. Lupercio, B. Uribe,
                         \textit{Gerbes over Orbifolds and Twisted K-Theory}, Commun. Math. Phys., 245 (2004), 449--489.

\bibitem{Mackenzie1987}
                        K. Mackenzie,
                        \textit{Lie groupoids and Lie algebroids
                        in Differential Geometry}.
                        Cambridge University Press, Cambirdge, 1987.

\bibitem{Moerdijk-Mrcun2003}
                        I. Moerdijk, J. Mrcun,
                        \textit{Introduction to Foliations and Lie Groupoids}.
                        Cambridge University Press, Cambridge, 2003.

\bibitem{Moerdijk-Pronk1997}   I. Moerdijk, D. A. Pronk,
                      \textit{Orbifolds, sheaves and groupoids},
                      K-Theory, 12 (1997), No. 1, 3--21.

\bibitem{Meyer-Zhu2015}
                        R. Meyer, C. Zhu,
                        \textit{Groupoids in categories with pretopology},
                        Theory Appl. Categ., 30 (2015), 1906--1998.

\bibitem{Pohl2017} A. D. Pohl,
\textit{The category of reduced orbifolds in local charts},
J. Math. Soc. Japan, 69 (2017), No. 2, 755--800.


\bibitem{Pronk1996}
                        D. A. Pronk,
                        \textit{\'Etendues and stacks as bicategories of fractions},
                        Compos. Math., 102 (1996), 243--303.

\bibitem{Schmeding2015} A. Schmeding,
\textit{The diffeomorphism group of a non-compact orbifold},
 Dissertationes Math., 507 (2015), 1--179.

\bibitem{Shen2014}
                       L. Shen,
                       \textit{Adjunctions in quantaloid-enriched categories},
                       PhD thesis, Sichuan University, 2014.

\bibitem{Tommasini2012} M. Tommasini, \textit{Orbifolds and groupoids}, Topology Appl., 159 (2012), 756--786.

\bibitem{Tommasini2016a}
                        M. Tommasini,
                        \textit{Some insights on bicategories of fractions: representations and compositions of 2-morphisms}, Theory Appl. Categ., 31 (2016), 257--329.

\bibitem{Tommasini2016b}
                        M. Tommasini,
                        \textit{A bicategory of reduced orbifolds from the point of view of differential geometry},
                        J. Geom. Phys., 108 (2016), 117--137.

\bibitem{Adem-Leida-Ruan2007}
                        A. Adem, J. Leida, Y. Ruan,
                        \textit{Orbifolds and stringy topology.}
                        Cambridge University Press, Cambirdge, 2007.

\bibitem{Borzellino-Brunsden2008}
             J. E. Borzellino, V. Brunsden,
             \textit{A manifold structure for the group of orbifold diffeomorphisms of a smooth orbifold}, J. Lie Theory, 18 (2008), 979--1007.

\bibitem{Chen-Du-Ono} B. Chen, C.-Y. Du, K. Ono,  \textit{Groupoid structure of moduli space of orbifold stable maps}, in preparation.

\bibitem{Chen-Ruan2002}
                      W. Chen, Y. Ruan,
                      \textit{Orbifold Gromov--Witten theory},
                      Cont. Math., 310 (2002), 25--86.

\bibitem{Chen-Ruan2004}
                        W. Chen, Y. Ruan,
                        \textit{A new cohomology theory for orbifold},
                        Commun. Math. Phys., 248 (2004), 1--31.

\bibitem{Du-Shen-Zhao2018}
                        C.-Y. Du, L. Shen, X. Zhao,
                        \textit{Spark complexes on good effective orbifold atlases categorically}, Theory Appl. Categ., 33 (2018), 784--812.


\bibitem{Leinster1998}
                        T. Leinster,
                        \textit{Basic bicategories},
                        arXiv:math/9810017, 1998.

\bibitem{Lupercio-Uribe2004} E. Lupercio, B. Uribe,
                         \textit{Gerbes over Orbifolds and Twisted K-Theory}, Commun. Math. Phys., 245 (2004), 449--489.

\bibitem{Mackenzie1987}
                        K. Mackenzie,
                        \textit{Lie groupoids and Lie algebroids
                        in Differential Geometry}.
                        Cambridge University Press, Cambirdge, 1987.

\bibitem{Moerdijk-Mrcun2003}
                        I. Moerdijk, J. Mrcun,
                        \textit{Introduction to Foliations and Lie Groupoids}.
                        Cambridge University Press, Cambridge, 2003.

\bibitem{Moerdijk-Pronk1997}   I. Moerdijk, D. A. Pronk,
                      \textit{Orbifolds, sheaves and groupoids},
                      K-Theory, 12 (1997), No. 1, 3--21.

\bibitem{Meyer-Zhu2015}
                        R. Meyer, C. Zhu,
                        \textit{Groupoids in categories with pretopology},
                        Theory Appl. Categ., 30 (2015), 1906--1998.

\bibitem{Pohl2017} A. D. Pohl,
\textit{The category of reduced orbifolds in local charts},
J. Math. Soc. Japan, 69 (2017), No. 2, 755--800.


\bibitem{Pronk1996}
                        D. A. Pronk,
                        \textit{\'Etendues and stacks as bicategories of fractions},
                        Compos. Math., 102 (1996), 243--303.

\bibitem{Schmeding2015} A. Schmeding,
\textit{The diffeomorphism group of a non-compact orbifold},
 Dissertationes Math., 507 (2015), 1--179.

\bibitem{Shen2014}
                       L. Shen,
                       \textit{Adjunctions in quantaloid-enriched categories},
                       PhD thesis, Sichuan University, 2014.

\bibitem{Tommasini2012} M. Tommasini, \textit{Orbifolds and groupoids}, Topology Appl., 159 (2012), 756--786.

\bibitem{Tommasini2016a}
                        M. Tommasini,
                        \textit{Some insights on bicategories of fractions: representations and compositions of 2-morphisms}, Theory Appl. Categ., 31 (2016), 257--329.

\bibitem{Tommasini2016b}
                        M. Tommasini,
                        \textit{A bicategory of reduced orbifolds from the point of view of differential geometry},
                        J. Geom. Phys., 108 (2016), 117--137.

\end{thebibliography}


\end{document}